\documentclass{article}

\usepackage[utf8]{inputenc}
\usepackage{authblk}
\usepackage{hyperref}
\usepackage{enumitem}

\usepackage{array}

\usepackage{verbatim}

\usepackage{tikz}
\usepackage{tikz-cd}
\usetikzlibrary{arrows}
\tikzcdset{arrow style=tikz,
diagrams={>={Stealth[round,length=4pt,width=6pt,inset=3pt]}}}

\usepackage{amsthm}

\usepackage{amsmath} 
\usepackage{amssymb}
\usepackage{bbold}
\usepackage{amsfonts}
\usepackage{faktor}
\usepackage[T1]{fontenc}
\usepackage{faktor}
\usepackage{xfrac}
\usepackage{dsfont}
\usepackage{color}
 \usepackage{todonotes}
\usepackage[english]{babel}
\usepackage{mathtools}

\newtheorem{theorem}{Theorem}
\newtheorem{definition}[theorem]{Definition}

\newtheorem{lemma}[theorem]{Lemma}

\newtheorem{proposition}[theorem]{Proposition}

\newtheorem{example}[theorem]{Example}
\newtheorem{remark}[theorem]{Remark}
\newtheorem{corollary}[theorem]{Corollary}

\newtheorem*{theorem*}{Theorem}

\DeclareMathOperator{\ord}{ord}
\DeclareMathOperator{\rk}{rk}

\DeclareMathOperator{\syz}{syz}
\DeclareMathOperator{\chara}{char}
\DeclareMathOperator{\jet}{jet}
\DeclareMathOperator{\coker}{coker}
\DeclareMathOperator{\Mat}{Mat}
\DeclareMathOperator{\ima}{ima}

\numberwithin{theorem}{section}

\renewcommand{\mod}{\,\operatorname{mod}\,}

\title{ADE Classification of Hypersurface Singularities over Local Rings}
\author{Yotam Svoray}
\date{}

\AtEndDocument{\bigskip{\footnotesize%
  \textsc{ Department of Mathematics, University of Utah, UT 84112} \par  
  \textit{E-mail address}: \texttt{svoray@math.utah.edu} 
}}

\begin{document}

\maketitle

\begin{abstract}
         In this paper we present an $ADE-$type classification of hypersurfaces of complete regular local rings based on their Cohen-Macaulay type. In order to preform this classification, we show how we can generalize the classical result regarding finite and countable Cohen-Macaulay type to larger cardinalities, depending on the cardinality of the residue field. 
\end{abstract}

  \tableofcontents

\section{Introduction}

The goal of this paper is to provide an analogue of the $ADE$ classification for hypersurfaces of complete regular local rings (including rings of mixed characteristic):

\begin{theorem*}[Generalized $ADE$ Classification]\label{thm:mainintro} \hfill

 Let $(R, \mathfrak{m}, \kappa)$ be a regular complete local ring of dimension $n\geq 1$ such that $\kappa$ is infinite with $\chara(\kappa)\neq 2$ and let $f \in \mathfrak{m}^2$. If $f$ has sparse Cohen-Macaulay type (see Definition~\ref{def:sparse}) then $f$ can be written as (see Definition~\ref{def:can_be_written}) either $ADE$ singularity over $R$, which are defined to be as follows (for $x_1, \dots, x_n$ a set of minimal generators of $\mathfrak{m}$):

\begin{center}
\hspace*{-0.8cm}
    \begin{tabular}{ | l | l | l |}
    \hline
    $ADE$  & Formula &  Note  \\ \hline \hline
    $A_k$ & $x_1^{k+1} + x_2^2 +\dots + x_n^2$ & $1 \leq k$  \\ \hline
    $D_k$ & $x_1\left(x_2^2  + x_1^{k-2}\right)+ x_3^2 +\dots + x_n^2$ & $4 \leq k$  \\ \hline
    $E_6$ & $x_1^3 + x_2^{4}+ x_3^2 +\dots + x_n^2$ &    \\ \hline
    $E_6^1$ & $x_1^3 + x_2^{4} + x_1^2 x_2^2+ x_3^2 +\dots + x_n^2$ & $\chara(\kappa) = 3$  \\ \hline
    $E_7$ & $x_1\left(x_1^2 +  x_2^{3}\right)+ x_3^2 +\dots + x_n^2$ &   \\ \hline
    $E_7^1$ & $x_1\left(x_1^2 +  x_2^{3} + x_1 x_2^2\right)+ x_3^2 +\dots + x_n^2$ & $\chara(\kappa) =3$  \\ \hline
    $E_8$ & $x_1^3 + x_2^{5}+ x_3^2 +\dots + x_n^2$ &   \\ \hline
    $E_8^1$ & $x_1^3 + x_2^{5} + x_1^2 x_2^3+ x_3^2 +\dots + x_n^2$ & $\chara(\kappa)=3$ \\ \hline
    $E_8^2$ & $x_1^3 + x_2^{5}  + x_1^2 x_2^2+ x_3^2 +\dots + x_n^2$ & $\chara(\kappa)=3$  \\ \hline
    $E_8^1$ & $x_1^3 + x_2^{5}+ x_1 x_2^4+ x_3^2 +\dots + x_n^2$ & $\chara(\kappa)=5$ \\ \hline 
    \end{tabular}
    \vspace{5mm}
\end{center}

or as one of the following (assuming $\kappa$ is uncountable as): 

\begin{center}
\hspace*{-0.8cm}
    \begin{tabular}{ | l | l |}
    \hline
    $A_\infty$ & $x_1^2+x_2^2 + \cdots +x_{n-1}^2$  \\ \hline
    $D_\infty$ & $x_1x_2^2+x_3^2+\cdots+ x_{n-1}^2$  \\ \hline
    \end{tabular}
    \vspace{5mm}
\end{center}
\end{theorem*}

The study of $ADE$ singularities over fields goes back to Arnold's $ADE$ classification of simple singularities over the complex numbers in~\cite{varchenko1985singularities}, which was based on the work of Artin in~\cite{artin1966isolated, artin11977coverings} and of Du Val in~\cite{du1934isolatedi, du1934isolatedii, du1934isolatediii} on rational surface singularities. Based on Arnold's work, Buchweitz, Greuel, Schreyer, and Kn\"{o}rrer in~\cite{buchweitz1987cohen, knorrer1986cohen} proved that the $ADE$ singularities over $\mathbb{C}$ are exactly the hypersurfaces of finite Cohen-Macaulay type in addition to the classification of countable Cohen-Macaulay type hypersurfaces as $A_\infty$ and $D_\infty$. Later, Greuel and Kr\"{o}ning in~\cite{greuel1990simple} and Greuel and Nguyen in~\cite{greuel2016right} proved an analogues result of an $ADE$ classification in the positive characteristic case (based on the concept of finite deformation type), which unlike the characteristic zero case, differs based on if we look at right or contact equivalence. These classifications are based on computations and analysis of hypersurface singularities using tools such as the Milnor and Tjurina numbers, equivalences, determinacy, etc. For more about the topology and algebra of complex hypersurfaces, see~\cite{greuel2007introduction, milnor2016singular, varchenko1985singularities}.  \\

When moving beyond the equi-characteristic case (power series over a field), we encounter problems that arise due to constraints of the ring. Specifically, in the mixed characteristic case, we have a problem that there are elements that are very rigid under automorphisms. This causes the groups of automorphisms to be much smaller compared to its field counterpart. For example, if $\text{Witt}(\kappa)[[x]]$ is the ring of power series in one variable over the ring of Witt vectors over an algebraically closed field $\kappa$ of characteristic $p$, then $p$ is the uniformizer of $\text{Witt}(\kappa)$ and thus stable under ring automorphisms. For a deeper review of Witt vectors, see~\cite{rabinoff2014theory} or Chapter II, Section 6 of~\cite{serre1979local}. Yet, there are still elements we wish to view as "the same" algebraically. For example, similar to Example 1.1. in~\cite{carvajal2019covers}, over $\text{Witt}(\kappa)[[x]]$ we see that $x^3+p^2$ and $x^2+p^3$ are similar even though their corresponding hypersurface rings are not isomorphic, as otherwise$\mod p$ they would be isomorphic over the field $\kappa$, which is impossible. For this goal, we introduce the notion of "can be written as", which we denote $\looparrowright$, with respect to which we preform our $ADE$ classification. In addition, it provides for us a characteristic free approach for the classification, that is, no matter if the ring is equi-characteristic or mixed-characteristic.\\ 

In most research on Cohen-Macaulay type of a singularity (as mentioned above), the discussion is limited to the cases where the Cohen-Macaulay type is at most countable. Yet, we present a new notion, which we call "sparse Cohen-Macaulay type", which compares the cardinality with that of the residue field. We show that many results on finite and countable Cohen-Macaulay type can be generalized to sparse Cohen-Macaulay type. For example, in Section~\ref{Sec:Uncountable} we show that Huneke-Leuschke-Takahashi theorem (first proposed by Schreyer for complex analytic local rings in~\cite{schreyer2006finite} and later proved by Huneke and Leuschke in~\cite{huneke2003local} for excellent rings and by Takahashi in~\cite{takahashi2007uncountably} for Cohen-Macaulay rings) is true in the sparse case, that is, if a Cohen-Macaulay local ring has Sparse Cohen-Macaulay type then it has at most a one dimensional singular locus. \\

In the past there has been some work on singularities and specifically $ADE$ singularities in mixed characteristic and over local rings. In the case of the ring of formal power series over mixed characteristic DVR, the $ADE$ singularities have been studied by Leuschke in~\cite{leuschke2000finite}, in which they proved that the $ADE$ singularities are indeed of Finite Cohen-Macaulay type. In the surface case (i.e. the ring of formal power series over a mixed characteristic DVR), $ADE$ singularities (assuming $\chara (\kappa) >5$) have been studied in~\cite{carvajal2019covers}. In that paper, inspired by the work of Lipman in~\cite{lipman1969rational}, the authors proved that the ADE surface singularities are exactly the rational surface singularities, similar to the results by Artin and Du Val. In addition, the Weyl monodromy of rational surface singularities in the mixed characteristic case has been studied in~\cite{kountouridis2023simple}.  Isolated singularities over local rings in the mixed characteristic setting have been studied in~\cite{li2024regular} (for local graded isolated singularities) and in~\cite{heitmann1994completions}. Mixed characteristic rings of finite Cohen-Macaulay type (that are not hypersurfaces) have been studied in the past, for example, in~\cite{puthenpurakal2016two}. Jeffries and Hochster in~\cite{hochster2021jacobian} presented a Jacobian criterion over unramified mixed characteristic, and this work was later generalized by KC in~\cite{kc2024singular} to the ramified case. \\

 Throughout this text, we assume that $(R, \mathfrak{m}, \kappa)$ is a local ring with maximal ideal $\mathfrak{m}$ such that $\kappa=\frac{R}{\mathfrak{m}}$ is infinite and $\chara(\kappa)\neq 2$ (with potentially additional assumption on $R$ added based upon the specific result stated). Given some $f \in R$, we define \textbf{the order} of $f$  to be $\ord(f) = \max\{ r \colon f \in \mathfrak{m}^r \}$. Regrading the $ADE$ singularities, if $\chara(\kappa) = 3$ (resp. $\chara(\kappa)=5$) we denote $E_k$ by $E_k^0$ for $k=6,7,8$ (resp. $E_8$ by $E_8^0 $). Given a homorophism $\Phi$, we denote its kernel by $\ker(\Phi)$, its cokernel by $\coker(\Phi)$, and its image by $\ima(\Phi)$.  When discussing infinite sums, we use multi-index notation (as in~\cite{greuel2007introduction}): If $\alpha = (\alpha_1, \dots, \alpha_n)$ is a tuple of non negative integers and $\underline{x}=(x_1, \dots, x_n)$ is a tuple of ring elements, we denote $\underline{x}^\alpha =  x_1^{\alpha_1} \cdots x_n^{\alpha_n}$ and we denote $\sum_{\underline{i}} a_{\underline{i}} \underline{x}^{\underline{i}} = \sum_{i_1, \dots, i_n \geq 0} c_{i_1, \dots, i_n} x_1^{i_1} \cdots x_n^{i_n}$, where $|\underline{i}|=i_1+\cdots +i_n$. \\

 \textbf{Acknowledgments.} This work was done as part of the Author's PhD thesis under the guidance of Karl Schwede, and we wish to thank him for his guidance, help, and support. We wish to thank Gert-Martin Greuel, Srikanth Iyengar, Graham Leuschke, Eugenii Shustin, and Tim Tribone for productive mathematical discussions and their inputs on some of the ideas. The author was partially support by NSF grant DMS-2101800.

\section{Uncountable Cohen-Macaulay Type}\label{Sec:Uncountable}

In this section we generalize the notions of finite and countable Cohen-Macaulay type to arbitrary cardinalities, based upon the cardinality of the residue field. For more information on finite and countable Cohen-Macaulay type, see~\cite{yoshino1990maximal, leuschke2012cohen}. 

\begin{definition}\label{def:cm_type}
    Let $(R, \mathfrak{m}, \kappa)$ be a local ring. We denote by $\mathcal{MCM}(R)$ the set of isomorphism classes $[M]$ of maximal indecomposable Cohen–Macaulay modules $M$ over $R$. If $|\mathcal{MCM}(R)|=\lambda$ we say that $R$ has \textbf{$\lambda$-Cohen-Macaulay type}.  
\end{definition}

The following lemma shows us that indeed Definition~\ref{def:cm_type} is well defined, and gives us an upper bound on the possible Cohen-Macaulay type. 

\begin{lemma}\label{lem:cardinality_cm}
    For every finite dimensional Noetherian local ring $(R, \mathfrak{m}, \kappa)$ (where $\kappa$ is infinite) we have that $\mathcal{MCM}(R)$ is a set of cardinality at most $|R|$. In addition, $|R| \leq 2^{|\kappa|}$.  
\end{lemma}

\begin{proof}
    Given some $[M] \in \mathcal{MCM}(R)$, then $M$ must be finitely generated, and so there exists some surjection $\varphi\colon R^{\oplus r} \to M$ for some $r>0$. Therefore we have that $\frac{R^{\oplus r}}{\ker \varphi} \cong M$. Thus, $\mathcal{MCM}(R)$ can be embedded inside the set $\{(R^{\oplus r}, N) \colon N \subset R^{\oplus r} \text{ submodule }\}$, which has cardinality at most $|R|$. Now, since $\dim(R)<\infty$, we can conclude that $\mathfrak{m}$ is a finitely generated ideal. Therefore, for every $i$ we have that $\frac{\mathfrak{m}^N}{\mathfrak{m}^{N+1}}$ is a finite dimensional vector space over $\kappa$, and thus has cardinality (at most) $|\kappa|$. Yet, the map $R \to \bigoplus_N \frac{\mathfrak{m}^N}{\mathfrak{m}^{N+1}}$ is an injection (as $\cap_{i=1}^\infty \mathfrak{m}^i =\{0\}$ by Krull's intersection theorem), and so we can conclude that $|R| \leq 2^{|\kappa|}$. 
\end{proof}

\begin{definition}\label{def:sparse}
We say that $(R, \mathfrak{m},\kappa)$ has a \textbf{sparse Cohen-Macaulay type} if $|\mathcal{MCM}(R)| < |\kappa|$. 
\end{definition}

\begin{remark}
\textup{Note that when $R$ does not have a sparse Cohen-Macaulay type, then by Lemma~\ref{lem:cardinality_cm} we must have that $ |\kappa| \leq |\mathcal{MCM}(R)| \leq 2^{|\kappa|}$. The question if there are any cardinalities $|\kappa| < \lambda < 2^{|\kappa|}$ (known as the generalized continuum hypothesis) is independent of the ZFC axiomatic system. For more information, see Chapter 1 of~\cite{fraisse2000theory}. }
\end{remark}

We now present a sparse version of the Huneke-Leuschke-Takahashi theorem. Recall that $\text{Sing}(R)=\{\mathfrak{p} \in \text{Spec}(R) \colon R_\mathfrak{p} \text{ is not regular}\}$. Our proof is inspired by Section 2 of~\cite{takahashi2007uncountably} and by Section 1 of~\cite{huneke2003local}, while the general idea was inspired by Remark 14.4 in~\cite{leuschke2012cohen}.

\begin{theorem}[Huneke-Leuschke-Takahashi]\label{thm:Huneke-Leuschke-Takahashi}
    Let $(R, \mathfrak{m}, \kappa)$ be a Cohen-Macaulay local ring. If $R$ has a sparse Cohen-Macaulay type then $\dim\left(\frac{R}{\mathfrak{p}}\right) \leq 1$ for every $\mathfrak{p} \in \text{Sing}(R)$. 
\end{theorem}

We start with a cardinal version of the prime avoidance theorem. It is inspired by the proof of Proposition 2.5 in~\cite{sharp1985baire}. For more information about different infinite versions of the prime avoidance theorem, see~\cite{chen2021infinite}. 

\begin{lemma}[Cardinal Prime Avoidance]\label{lem:prime_avoidance}
    Let $(R, \mathfrak{m}, \kappa)$ be a local ring, let $\mathfrak{a} \subset R$ be a finitely generated ideal, and let $S$ be a set of ideals such that $|S| < |\kappa|$. If $\mathfrak{a} \subseteq \bigcup_{I \in S} I$ then there exists some $\mathfrak{b} \in S$ such that $\mathfrak{a} \subseteq \mathfrak{b}$.  
\end{lemma}

\begin{proof}
    Let $\Lambda \subset R$ to be a complete set of lifts from $\kappa^\times$. That is, for every $c \in \kappa^\times$ there exists a unique $\tilde{c} \in \Lambda$ such that $\tilde{c} \mod \mathfrak{m} = c$. Note that for every $u \neq v \in \Lambda$ we have that $u-v$ is a unit in $R$. Let $\mathfrak{a}=\langle a_1, \dots, a_k \rangle$. For every $u \in \Lambda$, define $b_u = \sum_{i=0}^{k-1} u^i a_{i+1} \in \mathfrak{a}$. Note that $u \mapsto y_u$ defines a map $\Lambda \to \bigcup_{I \in S} I$. This induces a map $\Lambda \to S$ by sending $u$ to some $I_u$ such that $y_u \in I_u$. Since $|\Lambda|=|\kappa|>|S|$ then by the pigeonhole principle, there exists some $\Lambda_1 \subseteq \Lambda$ and some $\mathfrak{b} \in S$ such that $|\Lambda_1|=|\kappa|$ and $y_u \in \mathfrak{b}$ for every $u \in \Lambda_1$. Thus, there exists some $u_1, \dots, u_k \in \Lambda$ such that $b_{u_1}, \dots, b_{u_k} \in \mathfrak{b}$. If we look at the matrix $P \in \text{Mat}_k(R)$ defined by $[P]_{i,j} = u_{i}^{j-1}$, we have that $P$ is invertible (as its the Vandermonde matrix which has an invertible determinant), and $P(a_1, \dots, a_n) = (b_{u_1}, \dots b_{u_k})$. Thus we can conclude that $a_i \in \mathfrak{b}$ for every $i$, and so $\mathfrak{a} \subset \mathfrak{b}$.
\end{proof}

Recall that a set $T \subset \text{Spec}(R)$ is called closed under specialization if for every $\mathfrak{p} \in T$ and for every $\mathfrak{q} \supseteq \mathfrak{p}$ we have that $\mathfrak{q} \in T$. For more information on specialization, see~\cite[\href{https://stacks.math.columbia.edu/tag/0060}{Tag 0060}]{stacks-project}. 

\begin{lemma}\label{lem:Specialization_dim_1}
    Let $(R, \mathfrak{m}, \kappa)$ be a local ring and let $T$ be a set of prime ideals closed under specialization such that $|T| < |\kappa|$. Then $\dim\left(\frac{R}{\mathfrak{p}}\right) \leq 1$ for every $\mathfrak{p} \in T$. 
\end{lemma}

\begin{proof}
    Since $\mathfrak{m}$ is not contained in any non-maximal $\mathfrak{p} \in T$, then by Lemma~\ref{lem:prime_avoidance} we have that $\mathfrak{m}$ is not contained in the union of all non-maximal ideals in $T$. Therefore, there exists some $f \in \mathfrak{m}$ such that $f \notin \mathfrak{p}$ for any $\mathfrak{m} \neq \mathfrak{p} \in T$. Now, we claim that $\mathfrak{p}+\langle f \rangle$ is $\mathfrak{m}-$primary, since that gives us that $0=\dim\left( \frac{R}{\mathfrak{p}+\langle f \rangle}\right) \geq \dim\left( \frac{R}{\mathfrak{p}}\right)-1$. Given some $\mathfrak{q}$ such that $\mathfrak{p}+\langle f \rangle \subset \mathfrak{q}$, then we have that $\mathfrak{p} \subset \mathfrak{q}$. Therefore, since $T$ is closed under specialization, we have that $\mathfrak{q} \in T$. Yet, since $f \in \mathfrak{q}$, we can conclude that $\mathfrak{q}=\mathfrak{m}$. Therefore we must have that $\sqrt{\mathfrak{p}+\langle f \rangle} = \mathfrak{m}$, and the result follows.  
\end{proof}

\begin{proof}[Proof of Theorem~\ref{thm:Huneke-Leuschke-Takahashi}]
    First, note that by Proposition 2.3. in~\cite{takahashi2007uncountably} we have that 
    \begin{equation*}
        \text{Sing}(R) \subseteq \{\text{Ann} (\text{Tor}_1(M, M)) \colon M \in \mathcal{MCM}(R)\}. 
    \end{equation*}
    Therefore, since $R$ has sparse Cohen-Macaulay type, we get that $|\text{Sing}(R)|<|\kappa|$, and since $\text{Sing}(R) \subset \text{Spec}(R)$ is closed under specialization, the result follows from Lemma~\ref{lem:Specialization_dim_1}.
\end{proof}

\begin{remark}
    \textup{Note that if $R$ is an Excellent local ring with sparse Cohen-Macaulay type then $\textup{Sing} (R)$ is at most one dimensional, i.e., $\textup{Sing}(R)=V(I)$ with $\dim \left(\frac{R}{I}\right) \leq 1$.  }
\end{remark}

\section{Hypersurface case}\label{Sec:HYpersurface}

    From now onward, we focus on hypersurface singularities (i.e. quotients of a regular local ring $(R, \mathfrak{m})$ by a principal ideal) that have sparse Cohen-Macaulay type. In this section, we show that we need to focus on a small collection of hypersurfaces. 
    
    \begin{remark}
        \textup{Regarding hypersurfaces $\frac{R}{\langle f \rangle}$, when discussing the Cohen-Macaulay type of $f$ we mean the Cohen-Macaulay type of $\frac{R}{\langle f \rangle}$. }
    \end{remark}
    
    The following proposition is due to Buchweitz-Greuel-Schreyer (first presented as Corollary 1.7 in~\cite{buchweitz1987cohen}), and we recall it as it is used extensively in this section. 

    \begin{proposition}\label{prop:ideal}
        Let $(R, \mathfrak{m})$ be a regular local ring and let $f \in \mathfrak{m}$. Then there exists a surjective map 
        \begin{equation*}
            \mathcal{MCM}\left( \frac{R}{\langle f\rangle} \right) \twoheadrightarrow \mathcal{IS}(R,f) =  \{ I \subset R \colon f \in  I^2\}. 
        \end{equation*}
        In particular, if there are $|\kappa|$ many ideals $I \subset R$ such that $f \in I^2$ then $\frac{R}{\langle f\rangle}$ does not have a sparse Cohen-Macaulay type. 
    \end{proposition}

\begin{remark}
    \textup{Given a regular local ring $(R, \mathfrak{m}, \kappa)$, as in Definition~\ref{def:sparse}, we say that $f$ has \textbf{sparse ideal type} if $|\mathcal{IS}(R, f)| \leq |\kappa|$. Then Proposition~\ref{prop:ideal} tells us that if $f$ has a sparse Cohen-Macaulay type then it has sparse ideal type. For more information about ideal type (also known as "simplicity" in some literature) see Chapter 8 in~\cite{yoshino1990maximal} or Chapter 9.1 in~\cite{leuschke2012cohen}.  }
\end{remark}

 The proof of the following proposition is very similar to the proof presented in Lemma 9.3 and Lemma 9.4 in~\cite{leuschke2012cohen}, and is based upon the application of Proposition~\ref{prop:ideal}. Yet, we present it here for the sake of completeness and exactness. 

\begin{proposition}\label{prop:ord}
    Let $(R, \mathfrak{m})$ be a regular local ring and let $f \in \mathfrak{m}$. If $f$ has a sparse Cohen-Macaulay type then:
    \begin{enumerate}
        \item $\ord(f) \leq 3$.
        \item $\ord(f) \leq 2$ if $\dim(R) >1$ and $\kappa$ is algebraically closed.
        \item For every $g \in \mathfrak{m}$ we have that $f \notin \langle g^3 \rangle$.
        \item For every $\alpha, \beta \in \mathfrak{m}$ we have that $f \notin \langle \alpha, \beta^2 \rangle^3$.
    \end{enumerate}
\end{proposition}

\begin{proof}
As in the proof of Lemma~\ref{lem:prime_avoidance}, let $\Lambda \subset R$ to be a complete set of lifts from $\kappa^\times$. Again, recall that if $c_1, c_2 \in \Lambda$ then $c_1-c_2$ must be invertible. In each item, we assume towards contradiction and construct a collection of $|\kappa|$ ideals $\{I_u\}_{u}$ such that $f \in I_u^2$ for every $u$, which shows that $f$ can not have sparse ideal type, and so contradicts Proposition~\ref{prop:ideal}. 
    \begin{enumerate}
        \item If $\ord(f) \geq 4$ then we have that $f \in \mathfrak{m}^4$. Therefore, for any ideal $ \mathfrak{m}^2 \subset I \subset \mathfrak{m}^4$ we have that $f \in I^2$. Note that such ideals correspond to proper non-zero $\kappa-$vector subspaces of $\frac{\mathfrak{m}}{\mathfrak{m}^2}$, and so there are $|\kappa|$ such ideals.
        \item Assume towards contradiction that $\ord(f)>2$. By the previous item we must have that $ord(f) \leq 3$, and so we can conclude that $ord(f) = 3$. Denote the image of $f$ in $\frac{\mathfrak{m}^3}{\mathfrak{m}^4}$ by $\overline{f}$. Then $\overline{f}$ corresponds with a cubic form in the graded ring $\text{gr}_\mathfrak{m}(R)=\bigoplus_{i=0}^\infty \frac{\mathfrak{m}^i}{\mathfrak{m}^{i+1}} \cong \kappa[t_0,\dots t_d]$ where $\mathfrak{m} = \langle t_0, \dots, t_d \rangle$. Therefore the zero set $X$ of $\overline{f}$ in $\mathbb{P}_\kappa^d$ has cardinality $|\kappa|$. Fix some $\lambda=[\lambda_0 \colon \dots \colon \lambda_d] \in X$ and let $L_1, \dots, L_d$ be a basis for the $\kappa-$vector space of linear forms that vanish at $\lambda$. Then we have that
        \begin{equation*}
            \overline{f} \in \langle L_1, \dots, L_d \rangle \cdot \langle t_0, \dots, t_d \rangle^2\subset \text{gr}_\mathfrak{m}(R).
        \end{equation*}
        By lifting each $L_i \in \frac{\mathfrak{m}}{\mathfrak{m}^2}$ to some $\tilde{L}_i \in \mathfrak{m} \setminus\mathfrak{m}^2$, we can define $I_\lambda = \langle \tilde{L}_1, \dots, \tilde{L}_d \rangle\cdot R + \mathfrak{m}^2$. Note that $f \in \langle \tilde{L}_1, \dots, \tilde{L}_d \rangle \cdot \mathfrak{m}^2 + \mathfrak{m}^4 \subset I_\lambda^2$. Therefore if we look at $\{I_\lambda\}_{\lambda \in X}$ then if $\lambda_1 \neq \lambda_2$, then as they corresponds to different point in $X$, we can conclude that $I_{\lambda_1} \neq I_{\lambda_2}$. 
        \item Suppose there exists some $g \in \mathfrak{m}$ such that $f \in \langle g \rangle^3$. If $\frac{R}{\langle g \rangle}$ is not a DVR then there exists some $x,y \in \mathfrak{m}$ such that $\overline{x}=x \mod \langle g \rangle$ and $\overline{y} = y \mod \langle g \rangle$ are a part of a generating set of $\frac{\mathfrak{m}}{\langle g \rangle}$. Now, for every $ u \in \Lambda$ define $I_u = \langle x+yu, g \rangle$. Then we have that $f \in I_u^3 \subset I_u^2$. In addition, if $I_{u} = I_v$ but $ u \neq v$ then $(u-v)y \in I_u$. But since $u-v$ must be invertible, we get that $y \in I_u$ and so $x \in I_u$. Yet, then $\frac{I_u}{\langle g \rangle} = \langle \overline{x}, \overline{y} \rangle = \langle \overline{x} + \overline{y}u\rangle$, which contradicts the choice of $x$ and $y$. Now, if $\frac{R}{\langle g \rangle}$ is a DVR then we must have that $\dim(R)=2$ with $\ord(g)=1$. Thus, there exists some $ h \in \mathfrak{m}$ such that $\mathfrak{m}=\langle g,h \rangle$. For every $u \in \Lambda$ we define $I_u=\langle g +uh^2, gh\rangle$. Note that $g^3 \in I_u^2$ for every $u$, and so $f \in I_u^2$. Now, if $I_u = I_v$ for $u \neq v$ we have that $I_u = \langle g,h^2 \rangle$. Now, there exists some $a,b \in R$ such that $g=(g+uh^2)a+ghb$, and so $auh^2 = g-ga-ghb \in \langle g \rangle$, and since $u$ is a unit and $\mathfrak{m}=\langle g,h \rangle$, we can conclude that $a \in \langle g \rangle$. Thus there exists some $r \in R$ such that $a=gr$, which gives us that $gruh^2 = g-ga-ghb$, and so $1=(g+uh^2)r+hb \in \mathfrak{m}$, which is a contradiction. 
        \item Assume towards contradiction that there exists some $\alpha, \beta \in \mathfrak{m}$ such that  $f \in \langle \alpha, \beta^2 \rangle^3$.  For every $u \in \Lambda$ define the ideal $I_u =\langle \alpha + u\beta^2 , \beta^3 \rangle$. Then one can check that $f \in \langle \alpha, \beta^2 \rangle^3 \subset I_{u}^2$ for every $u \in \Lambda$. If $I_{u} = I_{v}$ with $u \neq v$ then we must have that  $\left(u - v\right) \beta^2 \in I_{u}$. If $u \neq v$ then we must have that $u-v$ is invertible. Thus $\beta^2 \in I_{u}$ and so $I_{u} = \langle \alpha, \beta^2\rangle$. Yet, we have that $\beta^2 \in \langle \alpha +u\beta^2, \beta^3 \rangle$ and so there exists some $s$ and $t$ such that $\beta^2 = s\left(\alpha + u\beta^2\right) + t\beta^3$. Therefore $\beta^2\left(1-t\beta\right) \in \langle \alpha +u\beta^2 \rangle$ and since $1-t\beta$ is a unit we can conclude that $\beta^2 \in \langle \alpha +u\beta^2 \rangle$. Hence we have that $I_{u} =\langle \alpha + u\beta \rangle$ which tells us that $f \in \langle \alpha + u\beta \rangle^3$, which contradicts the previous item.
    \end{enumerate}
\end{proof}

We can conclude from Proposition~\ref{prop:ord} that we are interested in the elements of $R$ that are of order $2$ (where $\dim(R) >2$). We start with the general structure of elements of order $2$. It is an analogue of the splitting lemma over a field (see, for example, Theorem 2.47 of~\cite{greuel2007introduction} and Lemma 3.9 of~\cite{greuel2016right}). In addition, it is a generalized analogue of Lemma 3.1 in~\cite{carvajal2019covers}. For more information on the splitting lemma, see~\cite{greuel2025splitting}.

\begin{definition}\label{def:power_series}
    Given some regular sequence $a_1, \dots, a_r$ in $\mathfrak{m}$ and some $g \in R$, we say that $g$ \textbf{is a power series} in $a_1, \dots, a_r$, and we denote $g=g(a_1, \dots, a_r)$, if there exists some $u_{\underline{i}} \in R$ that are either units or zero such that $g = \sum_{\underline{i}} u_{\underline{i}} \underline{a}^{\underline{i}}$. 
\end{definition}

\begin{remark}
    \textup{The definition of power series over a regular sequence can be viewed as a generalization of the fact that given a ring $A$, we have an embedding $A[[x_1, \dots, x_r]] \subset A[[x_1, \dots, x_s]]$ where $r <s$. Yet, in the general case, this notation is weaker because the units can be arbitrary. For example, if $(R, \mathfrak{m})$ is a two dimensional regular ring then $(1+y)x+(1+y^2)x^2$ is a power series in $x$ where $\mathfrak{m}=\langle x,y\rangle$.}
\end{remark}

\begin{theorem}[Splitting Lemma]\label{prop:split}
    Let $(R, \mathfrak{m}, \kappa)$ be a local ring and let $f$ be of order $2$. Then there exists some $\mathfrak{m} = \langle a_1, \dots, a_n \rangle$, some $1 \leq r \leq n$, and a sequence of units $u_1, \dots, u_r$ such that:
    \begin{equation*}
    f=u_1 a_1^2+\cdots +u_ra_r^2 +g \text{ for some $g \in \mathfrak{m}^3$.}
    \end{equation*}
    In addition:
    \begin{itemize}
        \item  If $R$ is Henselian and $\kappa$ is quadratically closed then we can choose $u_i=1$ for every $i$.
        \item If $R$ is complete and $\kappa$ is quadratically closed, then we can choose  $g$ to be a power series in $a_{r+1}, \dots, a_n$. 
    \end{itemize}     
\end{theorem}

\begin{proof}

First, let $b_1, \dots, b_n$ be a sequence that generates $\mathfrak{m}$ and let $f \in R$ be of order $2$. Then we can write $f=\sum_{|\underline{i}|\geq 2} u_{\underline{i}} \underline{b}^{\underline{i}}$, and set $f_2 = \sum_{|\underline{i}|= 2} u_{\underline{i}} \underline{b}^{\underline{i}}$ with $f_{>2}=f-f_{2}$. Then we can write $f_2 = \vec{b} H \vec{b}^{\intercal}$ for some matrix $H \in \text{Mat}_n(R)$ where $\vec{b}=(b_1, \dots, b_n)$. By Proposition 3.4 in Chapter I of~\cite{baeza2006quadratic} there exists some invertible matrix $W$ and some diagonal matrix $D$ such that $H = W D W^\intercal$.  Denote $\vec{w}=\vec{x}W=(w_1, \dots, w_{n})$, and note that $\mathfrak{m}=\langle w_1, \dots, w_n \rangle$ since $W$ is invertible.  In addition, denote the entries on the diagonal of $D$ by $u_1, \dots, u_k, u_{k+1}, \dots, u_{n}$ such that (without loss of generality) $u_1, \dots, u_k$ are units and $u_{k+1}, \dots, u_{n+1} \in \mathfrak{m}$. Then we have that
    \begin{equation*}
        f_2 = \vec{x} H \vec{x}^\intercal = \vec{x} (W D W^\intercal) \vec{x}^\intercal = u_1 w_1^2 + \cdots + u_n w_n^2.   
    \end{equation*}
This gives us that $f = u_1 w_1^2 + \cdots +u_k w_k^2 + g$, where $g = f_{>2} + u_{k+1} w_{k+1}^2 + \cdots + u_n w_n^2 \in \mathfrak{m}^3$, as desired. \\

Now, note that if $R$ is Henselian and $\kappa$ is quadratically closed then for every unit $u \in R$ there exists some $v$ such that $v^2 =u$. Therefore, if $f=u_1w_1^2+\cdots +u_kw_k^2+g$, then if we choose $v_i \in R$ such that $v_i^2=u_i$ and set $a_i=v_i w_i$ for every $i \leq k$, we have that $f = a_1^2+\cdots +a_k^2 +h$ for some $h \in \mathfrak{m}^3$. \\

Now, assume that $R$ is complete. Since every complete local ring is also Henselian (see, for example,~\cite[\href{https://stacks.math.columbia.edu/tag/04GM}{Tag 04GM}]{stacks-project} or Corollary 1.9 in~\cite{leuschke2012cohen}) and since $\kappa$ is quadratically closed, we can assume that $u_i=1$ for every $i \leq k$. So without loss of generality we can assume that $f = b_1^2+\cdots +b_k^2 +h$ for some $h \in \mathfrak{m}^3$. Then we can write 

    \begin{equation*}
        f = b_1^2 + \cdots + b_k^2 + f_3\left(b_{k+1}, \dots, b_n\right)  + \sum_{i=1}^k b_i g_i.
    \end{equation*}

    \noindent for some $f_3 \in  \mathfrak{m}^3$ that is a power series in $x_{k+1}, \dots, x_n$ and some $ g_1, \dots, g_k \in \mathfrak{m}^2$.  Now, by looking at $x_{i,1} \coloneqq b_i + \frac{g_i}{2}$ for $i=1,\dots, k$ (recalling that $\chara(\kappa)\neq 2$), we have that 

    \begin{equation*}
        f = x_{1,1}^2 + \cdots + x_{k,1}^2 + f_3(b_{k+1}, \dots, b_{n})  - \sum_{i=1}^k \frac{g_i^2}{4}.
    \end{equation*}

    \noindent Noting that $\sum_{i=1}^k \frac{g_i^2}{4} \in \mathfrak{m}^4$, we can write 

    \begin{equation*}
        f_3(b_{k+1}, \dots, b_{n})  - \sum_{i=1}^k \frac{g_i^2}{4}=f_4(b_{k+1}, \dots,b_n) + \sum_{i=1}^k x_{i,1}h_i,
    \end{equation*}

    \noindent where $f_4\left(b_{k+1}, \dots, b_{n}\right) \in \mathfrak{m}^3$ is a power series in $b_{k+1}, \dots, b_{n}$ with $f_4-f_3 \in \mathfrak{m}^4$ and $h_i \in \mathfrak{m}^3$ for every $i$. Therefore, by repeating this process, we get that for every $l$ there exists a sequence of generators $x_{1,l}, \dots, x_{k,l}$ such that
    
     \begin{equation*}
        f = x_{1,l}^2 + \cdots + x_{k,l}^2 + f_l\left(b_{k+1}, \dots, b_{n}\right)  + \sum_{i=1}^k x_{i,l} h_{l,i},
    \end{equation*}

    \noindent where $f_l \in \mathfrak{m}^3$ is a power series in $b_{k+1}, \dots, b_{n}$ with $f_l - f_{l-1} \in \mathfrak{m}^l$ and $h_{l,i} \in \mathfrak{m}^{l+1}$ for $i=1, \dots, k$. Since $x_{i,l+1}-x_{i,l} \in \mathfrak{m}^{l+1}$, we have that for every $i \leq k$ then sequence $\{x_{i,l}\}_{l=1}^\infty$ is Cauchy, and since $R$ is complete, it must converge to some $x_i$ as $l \to \infty$. In addition, since $h_{l,i} \in \mathfrak{m}^{l+1}$ we can conclude that $h_{l,i} \to 0$ as $l \to \infty$. Finally, since $f_l - f_{l-1} \in \mathfrak{m}^l$, we can conclude that  $ f_l\left(b_{k+1}, \dots, b_{n}\right) \to g(b_{k+1}, \dots, b_n)$ for some $g(b_{k+1}, \dots, b_n)$ that is a power series in $x_{k+1}, \dots, x_n$. Therefore we can conclude that 
    
    \begin{equation*}
        f = x_1^2 + \cdots + x_k^2  + g(b_{k+1}, \dots, b_n),
    \end{equation*}

    \noindent as desired.
\end{proof}

\begin{definition}
    In Proposition~\ref{prop:split}, The value $k=\rk(f)$ is called the \textbf{rank} of $f$. 
\end{definition}

\begin{remark}\label{prop:unique_rk}
 \textup{From the proof of Theorem~\ref{prop:split} we can observe that the rank of $f$ can be though of as the rank of the matrix $H \mod \mathfrak{m} \in \Mat_n(\kappa)$. Therefore, since rank of a mtrix is constant under congruence, we can conclude that $\rk$ is well defined. That is, if $a_1^2 + \cdots + a_r^2 +g = b_1^2 + \cdots + b_s^2 +h$ where $\mathfrak{m}=\langle a_1, \dots, a_n \rangle = \langle b_1, \dots, b_n \rangle$ and $g,h \in \mathfrak{m}^3$, we must have that $r=s$. Yet, we can not conclude that $g=h$. Take, for example, $R=k[[x,y]]$ where $k$ is any field, then we have that $x^2 + y^3 = (x+y^2)^2 + (y^3 - 2xy^2 - y^4)$ with $\mathfrak{m}=\langle x,y\rangle = \langle x+y^2,y \rangle$, but $y^3 \neq y^3 - 2xy^2 - y^4$.}
\end{remark}

Inspired by Proposition~\ref{prop:split}, we introduce the following notation which plays the role over regular local rings of change of variables (i.e. contact equivalence, as described in Section 1.2.1 in~\cite{boubakri2009hypersurface}). It is inspired by the classification in Theorem B of~\cite{carvajal2019covers}.

\begin{definition}\label{def:can_be_written}
    Let $(R, \mathfrak{m}, \kappa)$ be a regular local ring of dimension $n$ and let $f,g \in \mathfrak{m}$. We say that $f$ \textbf{can be written as} $g$, and we denote $f \looparrowright g$, if there exists some polynomial $H(z_1, \dots, z_n)$ whose coefficients are wither $0$ or $1$ and there exists some sequences $a_1, \dots, a_n$ and $b_1, \dots, b_n$ such that
    \begin{enumerate}
        \item $\mathfrak{m}=\langle a_1, \dots, a_n \rangle$, 
        \item $\mathfrak{m}=\langle b_1, \dots, b_n \rangle$,
        \item $\langle f\rangle = \langle H(a_1, \dots, a_n) \rangle$, 
        \item $\langle g\rangle = \langle H(b_1, \dots, b_n)\rangle.$
    \end{enumerate}
\end{definition}

\begin{example}\label{ex:A} Let $(R, \mathfrak{m}, \kappa)$ be a regular local ring of dimension $n$ and let $f \in \mathfrak{m}^2$. Then: 
    \begin{enumerate}
        \item $f \looparrowright A_k$ (for $k \geq 1$) if and only if there exists some units $u_1, \dots, u_n$ and some regular sequence $a_1, \dots, a_n$ such that $f = u_1a_1^2 + \cdots + u_{n-1}a_{n-1}^2 + u_n a_n^{k+1}$. 
        \item $f \looparrowright A_\infty$ if and only if there exists some units $u_1, \dots, u_{n-1}$ and some regular sequence $a_1, \dots, a_n$ such that $f = u_1a_1^2 + \cdots + u_{n-1}a_{n-1}^2 $.
    \end{enumerate}
    (Note that if $R$ is Henselian and $\kappa$ is quadratically closed (with $\chara (\kappa) \neq 2$), then we can assume that $u_i=1$ for every $i$, since we can take square roots of units in $R$).
\end{example}

\begin{remark}
\begin{enumerate}
    \item \textup{By Remark~\ref{prop:unique_rk} we can in fact conclude that the rank of $f$ are unique up to $\looparrowright$.}
    \item \textup{Note that by Lemma 4.2 in~\cite{greuel2018finite} (which, in turn, cites Proposition 2.3 in~\cite{mather1968stability}), we have that if $\mathfrak{m} = \langle a_1, \dots, a_n \rangle=\langle b_1, \dots, b_n \rangle$, then there exist an invertible matrix $U \in \Mat_n(R)$ such that $\vec{a}=U\vec{b}$, where $\vec{a}=(a_1, \dots, a_n)$ and $\vec{b}=(b_1, \dots, b_n)$. This, in particular, tells us that in the case where $R$ is complete and of equi-characteristic, then by the Cohen structure theorem (see~\cite{cohen1946structure}) we must have that $R \cong K[[\underline{x}]]$ for some field $K$, and so by the implicit function theorem over $K[[\underline{x}]]$ we have that $f \looparrowright g$ if and only if $f$ and $g$ are contact equivalent, i.e. there exists some automorphism $\varphi$ of $R$ such that $\langle f \rangle = \langle \varphi(f) \rangle$. For more information about contact equivalent, see Section 1.2.1 in~\cite{boubakri2009hypersurface} or Section I 2.1 in~\cite{greuel2007introduction}. }
\end{enumerate}
\end{remark}

The following proposition is a stronger version of the second item of Proposition~\ref{prop:ord} by taking into account the rank of hypersurfaces with sparse Cohen-Macaulay type. Its proof is inspired by the proof of Lemma 8.2 in~\cite{yoshino1990maximal}. 

\begin{proposition}\label{prop:n-2}
    Let $(R, \mathfrak{m}, \kappa)$ be a complete regular local ring of dimension $n>2$ where $\kappa$ is algebraically closed and let $f \in \mathfrak{m}^2$. If $f$ has sparse Cohen-Macaulay type then $\rk(f) \geq n-2$. 
\end{proposition}

\begin{proof}
    Since $n>2$, then by Propososition~\ref{prop:ord} we have that $\ord(f)=2$, and so by Theorem~\ref{prop:split} there exist some $\mathfrak{m}=\langle a_1, \dots, a_n \rangle$ such that $f=a_1^2+ \cdots+ a_r^2 + g(a_{r+1}, \dots, a_n)$ for some $g$ that is a power series in $a_{r+1}, \dots, a_n$, where $r=\rk(f)$. Assume towards contradiction that $r<n-2$. Since $g$ is a power series in $a_{r+1}, \dots, a_n$ then we can write $g=g_3+g_{>3}$ where $g_3$ is a homogeneous polynomial in $a_{r+1}, \dots, a_n$. More explicitly, if we write $g=\sum_{|\underline{i}|\geq 3} u_{\underline{i}} \underline{a}^{\underline{i}}$ where $u_{\underline{i}}$ is either a unit or zero, then we set $g_3=\sum_{|\underline{i}| = 3} u_{\underline{i}} \underline{a}^{\underline{i}}$ and $g_{>3}=g - g_3$. Observe that $g_{>3} \in \mathfrak{m}^4$. \\

    Define $G(x_{r+1}, \dots, x_n)=\sum_{|\underline{i}|\geq 3} v_{\underline{i}} \underline{a}^{\underline{i}} \in \kappa[x_{r+1}, \dots, x_n]$ where $v_{\underline{i}} = u_{\underline{i}} \mod \mathfrak{m}$ for every $ \underline{i}$. Note that $G$ is a homogeneous polynomial of degree $3$, and so we can look at $X=V(G) \subset \mathbb{P}_\kappa^{n-r-1}$. Since $n-r-1>1$ (as $n-2>r$) and since $\kappa$ is algebraically closed then $X \neq \emptyset$ must be infinite and so $|X|=|\kappa|$. Thus, given some $\lambda=[\lambda_{r+1} \colon \dots \colon \lambda_n] \in X$ we define 
    \begin{equation*}
        I_\lambda = \langle a_1, ..., a_r\rangle  +\langle c_ia_j - c_ja_j \colon  r+1 \leq i,j \leq n \rangle + \mathfrak{m}^2 \subset R,
    \end{equation*}
    where $c_i$ is a lift of $\lambda_i$ from $\kappa$ to $R$. Observe that if $\lambda \neq \mu \in X$ then $I_\lambda \neq I_\mu$ since $\frac{I_\lambda}{\mathfrak{m}^2}$, viewed as an ideal in $\text{gr}_\mathfrak{m}(R)$, defines the point $\lambda \in \mathbb{P}_\kappa^{n-r-1}$. \\

    Therefore, in order to finish the proof, it is enough to show that $f \in I_\lambda^2$, as then we can conclude that $\{I_\lambda\}_{\lambda \in X} \subset \mathcal{IS}(R,f)$, which contradicts Proposition~\ref{prop:ideal}. Note that we can assume that $\lambda = [0 \colon \dots \colon 0 \colon 1]$ since for every $\lambda \in X$ we can find a is a linear automorphism of $PGL(n-r-1)$, defined by an invertible matrix $A \in  \Mat_{n-r}(\kappa)$, that sends $\lambda$ to $[0 \colon \dots \colon 0 \colon 1]$, and we can lift $A$ to an invertible matrix $\tilde{A} \in \Mat_{n-r}(R)$ which gives us new generators $A \vec{a}$ of $\mathfrak{m}$ (together with $a_1, \dots, a_r$), where $\vec{a}=(a_{r+1}, \dots, a_n)$. Thus, if $\lambda = [0 \colon \dots \colon 0 \colon 1]$ then $I_\lambda = \langle a_1, \dots, a_{n-1}, a_n^2 \rangle$, and so $g_3 = \alpha a_n^3 \mod I_\lambda^2$, for some $\alpha \in R$ that is either a unit of zero. But, since $\lambda \in X$, then we can conclude that $\alpha \mod \mathfrak{m} = G_3(0, \dots, 0,1)=0$, and since $\alpha$ is either a unit or zero, we can conclude that $\alpha=0$ and so $g_3 \in I_\lambda^2$. Thus, as $g_{>3} \in \mathfrak{m}^4 \subset I_\lambda^2$, then we can conclude that $f=a_1^2+ \cdots +a_r^2 +g_3 + g_{>3} \in \langle a_1, \dots, a_n \rangle + I_\lambda^2 + \mathfrak{m}^4 \subset I_\lambda^2$, as desired. 
\end{proof}

Inspired by Proposition~\ref{prop:n-2}, we end this section with a classification of elements of order $2$ with rank either $n$ or $n-1$ over any regular local ring of dimension $n$ based on Example~\ref{ex:A} (with the classification of elements with rank $n-2$ over complete regular local rings preformed in the next section).\\

The following proposition is a classification of elements of full rank, which we can view as a version of Morse's lemma over local rings, as presented in Theorem 2.46 in~\cite{greuel2007introduction} (which in turn is an analytic version of Morse's lemma, as presented in~\cite{ morse1934calculus, varchenko1985singularities}).  Note that it is a generalization of Proposition 3.3. in~\cite{carvajal2019covers} (which can be thought of as a mixed characteristic version of Morse's lemma for surfaces). 

\begin{proposition}[Morse's Lemma]\label{prop:Morse}
    Let $(R, \mathfrak{m}, \kappa)$ be a regular local ring of dimension $n$ and let $f \in R$ be of order $2$. Then $\rk(f)=n$ if and only if $f \looparrowright A_1$.  
\end{proposition}

\begin{proof}
    If $f \looparrowright A_1$ then it is clear that $\rk(f)=n$. Now, assume that $\rk(f)=n$. Then by  Theorem~\ref{prop:split} we have that there exists some units $u_1, \dots, u_n$ and some $a_1, \dots, a_n$ that generate $\mathfrak{m}$ such that $f = u_1a_1^2 + \cdots+ u_na_n^2 +g $ for some $g \in \mathfrak{m}^3$. Therefore, we can write $g = h+ \sum_{i=1}^n a_i^2 b_i$ where $b_i \in \mathfrak{m}$ and $h=\sum_{i \neq j \neq k} b_{i,j,k} a_i a_j a_k$. Thus we get that $f = h+ \sum_{i=1}^n (u_i + b_i)a_i^2$, and since $u_i$ is a unit and $b_i \in \mathfrak{m}$ for every $i$, then $v_i=u_i+b_i$ is a unit as well. \\

    \noindent Without loss of generality, assume that $b_{1,2,3} \neq 0$. Then we define a matrix
     \begin{equation*}
         M = 
        \begin{bmatrix}
        v_1 & \frac{b_{1,2,3}a_3}{2}\\
        \frac{b_{1,2,3}a_3}{2} & v_2
        \end{bmatrix}
         \in \textup{Mat}_2(R).
     \end{equation*}
     We have that $\vec{w} M \vec{w}^\intercal = v_1a_1^2 + v_2a_2^2 + b_{1,2,3}a_1a_2a_3$, where $\vec{w}=(a_1, a_2)$. As in Theorem~\ref{prop:split},  by Proposition 3.4 in Chapter I of~\cite{baeza2006quadratic} there exists some invertible matrix $U$ and some diagonal matrix $D$ such that $M = U D U^\intercal$. Denote the values on the diagonal of $D$ by $w_1$ and $w_2$. Then we have that 
     \begin{equation*}
         v_1v_2 - \frac{b_{1,2,3}^2 a_3^2}{4}=\det(M)=\det(U D U^\intercal)=\det(U)^2 \det(D) = \det(U)^2 w_1w_2,
     \end{equation*}
    and since both $\det(M)$ and $\det(U)$ are units, we must have that $w_1$ and $w_2$ be invertible. Thus, if we define $U\vec{w}^\intercal=\vec{c}= (c_1, c_2)$, then we have that 
    \begin{equation*}
        v_1 a_1^2 +v_2 a_2^2 + v_3a_3^2 + b_{1,2,3}a_1a_2a_3 = v_3a_3^2 + \vec{w}M\vec{w}^\intercal = v_3a_3^2 + \vec{c}D \vec{c}^\intercal=w_1c_1^2 + w_2c_2^2 +v_3a_3^2. 
    \end{equation*}

    \noindent This gives us that 
    \begin{equation*}
        f = w_1c_1^2 + w_2c_2^2 +v_3a_3^2 + \sum_{i=4}^n v_ia_i + h_1
    \end{equation*}
    \noindent for some $h_1 = \sum b_{i,j,k} a_i a_j a_k$ where $i,j,k$ are different such that $i \neq 1, j\neq 2, k \neq 3$. Note that indeed $c_1, c_2, a_3, \dots, a_n$ generate $\mathfrak{m}$. Thus, by repeating this process over every $i,j,k$ such that $b_{i,j,k} \neq 0$, the result follows. 
\end{proof}

The following proposition is a classification of elements of rank $n-1$, which can be thought of as a generalization of Proposition 3.4 in~\cite{carvajal2019covers} and can be though of as an $\looparrowright$ analogue of Theorem 2.48. from~\cite{greuel2007introduction}.

\begin{proposition}\label{prop:A_k}
    Let $(R, \mathfrak{m}, \kappa)$ be a regular local ring of dimension $n$ and let $f \in R$ be of order $2$. If $\rk(f)=n-1$ then one of the following is true:
    \begin{itemize}
        \item $f \looparrowright A_k$ for some $k \in \{2, 3, \dots, \infty \}$.
        \item over $\hat{R}$ (the $\mathfrak{m}-$adic completion of $R$) we have that $f \looparrowright A_\infty$. 
    \end{itemize}
\end{proposition}

\begin{proof}
    Assume that $\rk(f)=n-1$. Then by Theorem~\ref{prop:split} we can write $f=u_1a_1^2+ \dots +u_{n-1}a_{n-1}^2+g$ for $g \in \mathfrak{m}^3$. If $g=0$ then we clearly have that $f \looparrowright A_\infty$. Otherwise, using the same technique as in Proposition~\ref{prop:Morse} with respect to $a_1, \dots, a_{n-1}$, then up to $\looparrowright$, we can assume that $g=a_n^2c_0$ for some $c_0 \in \mathfrak{m}$. We can write $c_0=\beta_1a_1 + \cdots +\beta_n a_n$. Thus, we get that
    \begin{equation*} 
    \begin{split}
        f &= u_1a_1^2+ \dots +u_{n-1}a_{n-1}^2 + a_n^2(\beta_1a_1 + \cdots +\beta_n c_n) \\
         & = u_1\left(a_1+\frac{a_n^2 \beta_1}{2u_1}\right)^2 + \cdots + u_{n-1}\left(a_{n-1}+\frac{a_n^2 \beta_{n-1}}{2u_{n-1}}\right)^2 + a_n^3\left( \beta_n - \frac{a_n\beta_1^2}{4u_1} - \cdots - \frac{a_n\beta_{n-1}^2}{4u_{n-1}} \right). 
    \end{split}
    \end{equation*}
    Therefore, if we define $a_{i,1}=a_i+\frac{a_n^2 \beta_i}{2u_i}$ for $i\leq n-1$, we get that $f=u_1a_{1,1}^2 + \cdots + u_{n-1}a_{n-1,1}^2 + c_1a_n^3$ for some $c_1 \in R$. If $c_1$ is a unit, then we can conclude that $f \looparrowright A_2$. Otherwise, we can repeat this process to get 
    \begin{equation*}
        f=u_{1}a_{1,k}^2 + \cdots +u_{n-1}a_{n-1, k}^2 + c_ka_{n}^k,
    \end{equation*}
    where $a_{i,k} - a_{i,k-1} \in \langle a_n^k \rangle$.  If $c_k$ is a unit for some $k$ then we get that $f \looparrowright A_{k+1}$. Otherwise, by moving to $\hat{R}$, we can conclude that $\{a_{i,k}\}_{k =1}^\infty$ is a Cauchy sequence, and therefore it converges to some $b_i \in \hat{R}$. Observe that indeed $b_1, \dots, b_{n-1}, a_n$ generate the maximal ideal  $\mathfrak{m}\hat{R}$ of $\hat{R}$.  Since $c_ka_{n}^k \to 0$ as $k \to \infty$, we can conclude that $f = u_1b_1^2 + \cdots +  u_{n-1}b_{n-1}^2 \in \hat{R}$, i.e. $f \looparrowright A_\infty$ as an element of $\hat{R}$. 
\end{proof}

\begin{remark}
    \textup{Note that if we assume that $R$ is complete then we can provide a simpler proofs for Proposition~\ref{prop:Morse} and Proposition~\ref{prop:A_k}. If $f$ has rank $n$ then in this case clearly $f \looparrowright x_1^2+ \cdots + x_n^2$. If $f$ has rank $n-1$ then we can write $f = a_1^2 + \cdots +a_{n-1}^2 +g$ where $g=\sum_{i=0}^\infty u_i a_n^i$ is a power series in $a_n$ (or 
    $g=0$ which gives us that $f \looparrowright A_\infty$), and therefore we can write $g=ua_n^k$ for some $k \geq 3$, which gives us that $f=a_1^2 + \cdots +a_{n-1}^2 + ua_n^k \looparrowright x_1^2 +\cdots +x_{n-1}^2 + x_n^k$.}
\end{remark}



\section{$ADE$ classification over $R$}\label{Sec:ADE}

We now turn to the~\nameref{thm:mainintro}, as it appears in the introduction (Recall that $(R, \mathfrak{m}, \kappa)$ is complete regular local ring whose quotient field $\kappa$ is algebraically closed in this case), in the same vein as Theorem B in~\cite{buchweitz1987cohen} and Theorem 1.4 in~\cite{greuel1990simple}.\\

By Proposition~\ref{prop:n-2}, it is enough to classify elements of order $2$ and rank at least $\dim(R)-2$ (in the case $\dim(R)>2$). Note that by Proposition~\ref{prop:Morse} and by Proposition~\ref{prop:A_k} we can completely classify elements of rank $\dim(R)$ and $\dim(R)-1$, respectively. Therefore, it is enough to understand the case where $\rk(f) = \dim(R)-2$, in which case by Theorem~\ref{prop:split} we can write $f \looparrowright x_1^2 + \cdots + x_{\dim(R)-2}^2 + g(x_{\dim(R)-1}, x_{\dim(R)})$ for some $g$ of order $3$ that is a power series in $_{\dim(R)-1}$ and $ x_{\dim(R)}$ with $\mathfrak{m}=\langle x_1, \dots, x_{\dim(R)}\rangle$. Therefore, our goal for this section is to classify (up to $\looparrowright$) elements of order $3$ over $R$ with $\dim(R)=2$ with $\mathfrak{m}=\langle x, y \rangle$ (as a similar computation can be preformed for power series over a regular sequence with two elements). \\

We start by understanding the order $3$ "component" of $f$. This is a version of Proposition 2.50 in~\cite{greuel2007introduction} and of Proposition 9.2.12 in~\cite{de2013local}, and a generalization of Lemma 3.5. in~\cite{carvajal2019covers}. 

\begin{lemma}\label{lem:jet_3_2}
    Let $(R, \mathfrak{m}, \kappa)$ be a $2$ dimensional regular local ring where $\kappa$ is algebraically closed field. Let $f \in R$ be of order $3$. Then there exists some $g \in \mathfrak{m}^4$ and some $x,y$ that generate $\mathfrak{m}$ such that 
    \begin{enumerate}
        \item $f-g = xy\left(x+y\right)$,
        \item $f-g = x^2y$,
        \item $f-g = x^3$,
    \end{enumerate}
\end{lemma}

\begin{proof}
    Since $f$ is of order $3$ then for $\mathfrak{m}=\langle x,y\rangle$, there exists some $g \in \mathfrak{m}^4$ such that we can write $f -g = ay^3 + by^2x+ cyx^2 + dx^3$ for some $a,b,c,d\in R$ that are either units or zero. So we classify $ay^3 + by^2x+ cyx^2 + dx^3$ based on which of $a,b,c,d$ are units and which are zeros.\\
    
    If $a=d=0$ then we can write $by^2x+ cyx^2 = y x \left(by + cx\right)$. If $b=c=0$ then we get that $f-g=0$, which is impossible since $f \notin \mathfrak{m}^4$. If only one of them is zero, then without loss of generality $c=0 \neq b$, and so $f-g=y x b y = y^2x_1$ for $x_1=bx$. If $b,c \neq 0$ then we have that $y x \left(by + cx\right)= b y x \left(y + \frac{c}{b} x\right) \looparrowright yx\left(y+x\right)$. \\

    If either $a \neq 0$ or $d\neq 0$ we can assume without loss of generality that $a \neq 0$. Since we are working up to units, then we have that $f-g = a(y^3 + (a^{-1}b)y^2x + (a^{-1}c)yx^2 + (a^{-1})dx^3)$, and so we can assume that $a=1$. Now, set $Q\left(t\right)=t^3+bt^2 +ct+d \in R[t]$ and set $q(t) = Q(t) \mod \mathfrak{m} \in \kappa[t]$. Observe that $f-g = x^3 Q(\frac{y}{x})$.  Since $\kappa$ is an algebraically closed field, then $q(t)$ splits into a product $q\left(t\right)=\left(t-\alpha\right)\left(t-\beta\right)\left(t-\gamma\right)$. Therefore, since $q(t) = Q(t) \mod \mathfrak{m}$, we can conclude for some lifts $\alpha_1, \beta_1, \gamma_1 \in R$ of $\alpha,\beta, \gamma \in \kappa$, we have that $Q(t)=(t-\alpha_1)(t-\beta_1)(t-\gamma_1)+q_1(t)$ for some $q_1(t)=\theta_2t^2 + \theta_1t+\theta_0$ with $\theta_0, \theta_1, \theta_2 \in \mathfrak{m}$. Thus, we can conclude that $ f-g = x^3Q(\frac{y}{x}) =  \left(y - \alpha_1 x\right)\left(y - \beta_1 x\right)\left(y - \gamma_1 x\right) +x^3q_1(\frac{y}{x})$. Note that since  $x^3q_1(\frac{y}{x}) =  \theta_2y^2x+\theta_1 yx^2 + \theta_0x^3$ with $\theta_0, \theta_1, \theta_2 \in \mathfrak{m}$, then $x^3q_1(\frac{y}{x}) \in \mathfrak{m}^4$ and so for $g_1=g+x^3q_1(\frac{y}{x})$ we have that $f-g_1=\left(y - \alpha_1 x\right)\left(y - \beta_1 x\right)\left(y - \gamma_1 x\right)$. Then there are three cases: 

    \begin{enumerate}
        \item If $\alpha_1 = \beta_1=\gamma_1$ then for $y_1=y-\alpha_1x$ we have that $f-g_1 = y_1^3$ with $\mathfrak{m}=\langle x, y_1 \rangle$.

        \item If, without loss of generality, $\alpha_1=\beta_1 \neq \gamma_1$ then by setting $x_1=y-\alpha_1x$ and $y_1=y-\gamma_1x$ then $f-g_1=y_1^2x_1$. So:
        \begin{itemize}
            \item If  $\alpha_1-\gamma_1 \in \mathfrak{m}$  then 
            \begin{equation*}
                f-g_1 = (y-\alpha_1x)^2(y-\gamma_1x)=(y-\alpha_1x)^3-(y-\alpha_1x)^2x(\gamma_1-\alpha_1),
            \end{equation*}
             \noindent which reduces to the previous case, as $x(\gamma_1-\alpha_1)(y-\alpha_1x) \in \mathfrak{m}^4$.
            \item If $\alpha_1-\gamma_1 \notin \mathfrak{m}$, i.e. is invertible, then $\mathfrak{m}=\langle x_1, y_1 \rangle$ since we have that $x=\frac{\gamma_1x_1-\alpha_1x_1}{\gamma_1-\alpha_1}$ and $y=\frac{y_1-x_1}{\gamma_1-\alpha_1}$.
        \end{itemize}
          \item If $\alpha_1, \beta_1, \gamma_1$ are all different then:
          \begin{itemize}
              \item If, without loss of generality, $\alpha_1-\beta_1 \in \mathfrak{m}$ then
              \begin{equation*}
             f-g_1=(y-\alpha_1x)^2(y-\gamma_1 x)-(y-\alpha_1x)(y-\gamma_1 x)x(\beta_1-\alpha_1),
            \end{equation*}
            which reduces to the previous case. 
            \item If $\alpha_1-\beta_1, \beta_1-\gamma_1, \gamma_1-\alpha_1$ are all units, then we have that 
            \begin{equation*}
                (y-\gamma_1x)=\left(\frac{\gamma_1-\beta_1}{\alpha_1-\beta_1}\right)(y-\alpha_1x) + \left(\frac{\alpha_1-\gamma_1}{\alpha_1-\beta_1}\right)(y-\beta_1x), 
        \end{equation*}
              and since both $\frac{\gamma_1-\beta_1}{\alpha_1-\beta_1}$ and $\frac{\alpha_1-\gamma_1}{\alpha_1-\beta_1}$ are units, then by setting $x_2=\left(\frac{\gamma_1-\beta_1}{\alpha_1-\beta_1}\right)(y-\alpha_1x)$ and $y_2=\left(\frac{\alpha_1-\gamma_1}{\alpha_1-\beta_1}\right)(y-\beta_1x)$ we can conclude that $f-g_1=\frac{(\alpha_1-\beta_1)^2}{(\gamma_1-\beta_1)(\alpha_1-\beta_1)}x_2y_2(x_2+y_2) \looparrowright xy(x+y)$.
          \end{itemize}
    \end{enumerate}  
\end{proof}

\begin{definition}
    For a given $f$ of order $3$, the term from $\{xy\left(x+y\right), x^2y, y^3\}$ of which $f-g$ can be written as (via Lemma~\ref{lem:jet_3_2}) is called the \textbf{$3-$jet of $f$} and is denoted by $\jet_3(f)$. 
\end{definition}

\begin{remark}
    \textup{The term “jet” is used here as an analogy to the notion of a jet of a power series—that is, a Taylor polynomial of a given degree. For further details, see Section 2.2 in~\cite{greuel2007introduction} or Section 3.1 in~\cite{boubakri2009hypersurface}.}
\end{remark}

Therefore, our goal is to classify $f$ based upon the value of $\jet_3(f)$, as presented in Lemma~\ref{lem:jet_3_2}. We start with $x_1x_2\left(x_1+x_2\right)$ and $x_1^2x_2$, which correspond to $D_k$ for $k \in \{4, \dots, \infty\}$. Their classification will rely on the following Lemma. 

\begin{lemma}\label{lem:D_k_det}
    Let $(R, \mathfrak{m}, \kappa)$ be a $2-$dimensional complete regular local ring such that $\kappa$ is quadratically closed. Let  $k \geq 4$ and let $g \in \mathfrak{m}^{k+1}$. Then $y^2x+x^k+g \looparrowright y^2x+x^k$.
\end{lemma}

\begin{proof}

    First, we claim that if $u \in R$ is a unit and $r>k>2$ then $  y^2x +x^k + u y^r \looparrowright   y^2x +x^k$. Then by setting $x_1= x+u y^{r-2}$ we can write 
\begin{equation*}
      y^2x +x^k + u y^r =   y^2\left(x+u y^{r-2}\right) +x^k =   y^2x_1 +\left(x_1-u y^{r-2}\right)^k. 
\end{equation*}
Therefore, there exists some $a \in \mathfrak{m}$ and some unit $v$ such that 
\begin{equation*}
   y^2x_1 +\left(x_1-u y^{r-2}\right)^k= y^2x_1\left(1+a\right)+ x_1^k +v y^{\left(r-2\right)k} = y_1^2x+ x^k+u_1 y_1^{\left(r-2\right)k} 
\end{equation*}
 for some unit $u_1$ and for $y_1=y(1+a)^{\frac{1}{2}}$ (noting that the square root of $1+a$ exists since $R$ is Henselian and quadratically closed). By repeating this process we get a sequence  $\{ (x_i, y_i) \}_{i=1}^\infty$ such that $\mathfrak{m} = \langle x_i, y_i \rangle$ and $y^2x +x^k + u y^r=y_i^2x_i+x_i^k+u_i y_i^{r_i}$ for a sequence of units $\{u_i\}_{i=1}^\infty$ and for an increasing sequence of positive integers $r_i \to \infty$. In addition, note that $x_i -x_{i+1}, y_i- y_{i+1} \in \mathfrak{m}^{r_i}$. Thus the sequences $\{x_i\}_{i=1}^\infty$ and $\{y_i\}_{i=1}^\infty$ are Cauchy and since $R$ is a complete ring then $\{(x_i,y_i)\}_{i=1}^\infty$ converges to some $(x_0,y_0)$. Since $u_i y_i^{r_i} \to 0$ as $i \to \infty$, we get that $y^2x +x^k + u y^r  = y_0^2x+x_0^k$. \\

Now, for the general case, let $g \in \mathfrak{m}^{k+1}$. We can write $g=a y^k+b y^2x + cx^k$ for some $a \in R$, $b \in \mathfrak{m}^{k-3}$, and $c \in \mathfrak{m}$. Therefore, we can write
\begin{equation*}
    y^2x+x^k+g =   y^2x\left(1+b\right) +x^k\left(1+c\right)+a y^k \looparrowright   y^2x +x^k+a_1 y^k,
\end{equation*}

\noindent for some $a_1 \in R$. If $a_1$ is a unit, then we are done by the first part of the proof. If not, we can write $a_1= y\alpha_1+x\beta_1$, and so 
\begin{equation*}
    y^2x +x^k+a_1 y^k=y^2x\left(1+\beta_2 y^{k-2}\right) +x^k+\alpha_1 y^{k+1} \looparrowright   y^2x +x^k+a_2 y^{k+1}
\end{equation*}
 for some $a_2$. Now, if $a_2$ is a unit then we are done by the first part. Otherwise we can continue this process to get a collection $\{a_r\}_{r=1}^\infty$ such that $y^2 x +x^k +g \looparrowright y^2x +x^k+a_2 y^{k+r}$. If there exists some $r$ such that $a_r$ is a unit $y^2x+x^k+g  \looparrowright   y^2x +x^k + a_r y^{k+r}$, which reduces to the previous case.\\

If $a_r$ is not a unit for any $r$, then as in the first part, 
we get a sequence  $\{ (x_i, y_i) \}_{i=1}^\infty$   such that $\mathfrak{m} = \langle x_i, y_i \rangle$ and $y^2 x +x^k +g = y_i^2x +x_i^k+a_2 y_i^{k+i}$. Since for every $i_1 < i_2$ we have that $x_{i_1}-x_{i_2}, y_{i_1}-y_{i_2} \in \mathfrak{n}^{k+i_1}$, we get that $\{(x_i, y_i)\}_{i=1}^\infty$ is a Cauchy sequence of elements in $R$. Yet, since $R$ is a complete ring, then $(x_i,y_i) \to (x_0, y_0)$. Therefore, we can conclude that $y^2 x +x^k +g = y_0^2x +x_0^k$ as desired. 
\end{proof}

\begin{remark}
    \textup{Lemma~\ref{lem:D_k_det} (and in fact, Proposition~\ref{prop:Morse} and Proposition~\ref{prop:A_k} as well) can be thought of as an example of determinacy, inspired by the field case as explained in Section 3.1 of~\cite{boubakri2009hypersurface} (which in turn are based on Section 2.2 of~\cite{greuel2007introduction} in the characteristic zero case, which were expanded upon in Section 2 of~\cite{boubakri2012invariants} and in~\cite{greuel2019finite}). For more information on determinacy of general group actions, see~\cite{boix2022pairs}. In addition, determinacy over local rings that are algebras over an algebraically closed field of characteristic zero have been studied by Belitskii and Kerner in~\cite{belitskii2012matrices, belitskii2016finite, belitskii2019finite}.}
\end{remark}

\begin{proposition}
    Let $(R, \mathfrak{m}, \kappa)$ be a $2-$dimensional complete regular local ring where $\kappa$ is algebraically closed field and let $f \in R$ be of order $3$. Then: 
    \begin{enumerate}
        \item  $\jet_3(f) = xy(x+y)$ if and only if $f \looparrowright D_4$.
        \item $\jet_3(f) = x^2y$ if and only if $f \looparrowright D_k$ for $k \in \{5, \dots, \infty\}$. 
    \end{enumerate}
\end{proposition}

\begin{proof}
    First, if $\jet_3(f) = xy(x+y)$ then there exists some $a$ and $b$ that generate $\mathfrak{m}$ and some $g \in \mathfrak{m}^4$ such that $f = b^2 a + a^3 + g$. Then by Lemma~\ref{lem:D_k_det} we must have that $f \looparrowright b^2 a + a^3$, which is $D_4$. \\

    Now, assume that $\jet_3(f) = x^2y$. Then we can find some $a$ and $b$ that generate $\mathfrak{m}$ such that $f=a^2b+g$ for some $g \in \mathfrak{m}^4$. If $g=0$ then clearly we have that $f \looparrowright D_\infty$. Otherwise, let $k \geq 4$ be the smallest integer such that $g \in \mathfrak{m}^k$. Thus, we can write 
    \begin{equation*}
        f-g=b^2a+\alpha_m  b^k + \beta  ab^{k-1}+ b^2h\left( a\right)
    \end{equation*}
    for some units $\alpha_m, \beta \in R$ and for some $h \in \mathfrak{m}^{k-2}$ that is a power series in $a$. By setting $ y_1 = b-\frac{\beta b^{k-2}}{2}$ and $x_1=a-h\left(a\right)$, we have that $f =   y^2x + \alpha_{k} x^k +g_1$ for some $\alpha_{k}$ and some $g_1 \in \mathfrak{m}^{k+1}$. Thus, by repeating this process we get a sequence $\{\alpha_{r}\}_{r>k}$ and sequence of pairs $\{(a_r, b_r)\}_{r>k}$ such that for every $r$ we have that $\mathfrak{m} = \langle a_r, b_r \rangle$ and there exists some $g_r \in \mathfrak{m}^{r+1}$ such that  $f-g_r =    b_r^2a_r + \alpha_k a_r^k$.\\
    
    If there exists some $r_0$ such that $\alpha_{r_{0}}$ is a unit then, up to increasing $k$ to $r_0$, we can assume that $\alpha_k$ is a unit. Therefore, we get that $f =  \alpha_k(  (\alpha_k^{-\frac{1}{2}}b_r)^2a_r +  a_r^k+h_r$ for some $h_r \in \mathfrak{m}^{k+1}$.  Thus, by applying Lemma~\ref{lem:D_k_det} we can conclude that $f \looparrowright y^2x + x^k$. \\

    Otherwise, if for every $r$ we have that $\alpha_r \in \mathfrak{m}$, then we can conclude that there exists some sequence $\{(x_r, y_r)\}_{r>k}$ such that $f = y_r^2x_r +\alpha_rx_r^r + g_r$ for some $g_r \in \mathfrak{m}^r$. Thus, by taking $r \to \infty$, and observing that both $\{x_r\}_{r>k}$ and $\{y_r\}_{r>k}$ are both Cauchy sequences in a complete ring $R$, we can conclude that $x_r \to x_0$ and $y_r \to y_0$ such that $f=y_0^2x_0 \looparrowright D_\infty$. 
\end{proof}

Finally, in order to finish the $ADE$ classification it is enough to classify (up to $\looparrowright$) the elements $f$ for which $\jet_3(f) = x^3$. Note that up until now, in the computations and classifications of $A_k$ and of $D_k$, we did not care about the value of $\chara(\kappa)$ (except for the fact that it is not $2$). Yet, as we see now in the classification of $E_k$, the value of $\chara(\kappa)$ matters significantly. Specifically, for $\chara(\kappa)=3$ and $\chara(\kappa)=5$ we get additional forms. This, in fact corresponds to their corresponding positive characteristic cases, as can be seen in~\cite{greuel1990simple}. The proof is inspired by the computations done in Section 1 of~\cite{kiyek1985einfache}.

\begin{proposition}\label{prop:splitting_n=2}
    Let $(R, \mathfrak{m}, \kappa)$ be a $2-$dimensional Henselian regular local ring where $\kappa$ is algebraically closed. Let $f \in R$ be of order $3$ with $\jet_3\left( f\right) = y^3$. Then either $f \in \langle  x_1,  x_2^2 \rangle^3$ for some $x_1,x_2 \in \mathfrak{m}$ or $f$ can be written as one of the following forms:
    \begin{enumerate}[label=(\arabic*)]
        \item $   x_1^3+  x_2^4$, \label{Eq1}
        \item $   x_1^3+  x_2^4 +  x_1^2 x_2^2$, \label{Eq2}
        \item $   x_1^3+  x_1 x_2^3$, \label{Eq3}
        \item $   x_1^3+  x_1 x_2^3 +  x_1^2 x_2^2$, \label{Eq4}
        \item $   x_1^3+  x_2^5$, \label{Eq5}
        \item $   x_1^3+  x_2^5+ x_1^2 x_2^2$, \label{Eq6}
        \item $   x_1^3+  x_2^5+ x_1^2 x_2^3$, \label{Eq7}
        \item $   x_1^3+  x_2^5+ x_1 x_2^4$. \label{Eq8}
    \end{enumerate}

    In addition, if $R$ is Henselian and $\kappa$ is algebraically closed, we have the following equivalences:
    \begin{itemize}
        \item If $\chara(\kappa) \neq 3$ we have that \ref{Eq2} can be written as \ref{Eq1}, \ref{Eq4} can be written as \ref{Eq3}, \ref{Eq6} can be written as \ref{Eq8}, and \ref{Eq7} can be written as \ref{Eq5}.
        \item If $p\neq 5$ then \ref{Eq8} can be written as \ref{Eq5}.
    \end{itemize}
       
\end{proposition}

Due to the length of the proof of Proposition~\ref{prop:splitting_n=2}, we split it into a collection of lemmata. 

\begin{lemma}\label{lem:jet_3-form}
    Let $f \in R$ be of order $3$ with $\jet_3\left( f\right) = x^3$. Then 
    \begin{equation*}
        f =  \theta  y^3 + a\left( x\right) y^2  x^2 + b\left( x\right) y x^3 + c\left( x\right) x^4
    \end{equation*}
    for some $a(x),b(x),c(x)$ power series in $x$ and some unit $\theta$, where $\mathfrak{m}=\langle x, y \rangle$. 
\end{lemma}

\begin{proof}
    We can write $f =   y^3 +g$ for some $g \in \mathfrak{m}^4$ with $\mathfrak{m}=\langle x,y \rangle$. Therefore we can write $g\left( x\right)=\sum_{i=0}^\infty a_i\left( x\right)  y^i$ for some $a_i\left( x\right)$ that are all power series in $x$. That is, since we can write $g\left( x\right)= \sum_{i=0}^\infty (\sum_{j=0}^\infty u_{i,j} x^j) y^i$ where $u_{i,j}$ are either units or zero, then we can set $a_i(x)=\sum_{j=0}^\infty u_{i,j} x^j$ which are power series in $x$ (as in Definition~\ref{def:power_series}).  Therefore, since we can write $y^3 +\sum_{i=4}^\infty a_i(x)y^i = y^3 ( 1+ \sum_{i=1}^\infty a_{i+3}(x) y^{i})$,  then for the unit $\theta =( 1+ \sum_{i=1}^\infty a_{i+3}(x) y^{i})$ we can write

 \begin{equation*}
      f =   \theta  y^3 + a_0\left( x\right) + a_1\left( x\right) y +a_2\left( x\right)  y^2.
 \end{equation*}

\noindent Yet, since $g \in \mathfrak{m}^4$ we must have that $a_0\left( x\right) \in \mathfrak{m}^4$, $a_1\left( x\right) \in \mathfrak{m}^3$ and $a_2\left( x\right) \in \mathfrak{m}^2$. Thus we can write $a_0\left( x\right)= x^4 c\left( x\right)$, $a_1\left( x\right)= x^3b\left( x\right)$, and $a_2\left( x\right)= x^2a\left( x\right)$  for some $a,b,c$ that are power series in $x$ and the result follows. 
\end{proof}

\begin{lemma}[$E_6$ case]
    In the settings of Lemma~\ref{lem:jet_3-form}, if $\ord\left(c\right)=0$ then $f$ can be written as  either $y^3+x^4$ or $y^3 +  y^2  x^2 +  x^4$.
\end{lemma}

\begin{proof}
 We have the following cases:

    \begin{enumerate}
        \item \underline{If $\ord\left(a\right) \geq 2$} then there exists some  $a_1(x)$ that is a power series in $x$ such that $a(x)=x^2a_1(x)$. Thus
        \begin{equation*}
           f =  \theta  y^3  + b\left( x\right) y x^3 + (c(x)+ a_1( x) y^2)x^4. 
        \end{equation*}
        Since $c(x)$ is a unit we have that $c(x)+ a_1( x)y^2$ is a unit as well. Therefore:
        \begin{enumerate}
            \item \underline{If $b(x)$ is not a unit} then we can write $b(x)=b_1(x)x$ for some $b_1(x)$ that is a power series in $x$. Therefore we can conclude that 
            \begin{equation*}
           f =  \theta  y^3  + (c(x)+ b_1(x)y+ a_1( x) y^2)x^4,  
        \end{equation*}
        \noindent and since $u=c(x)+ b_1(x)y+ a_1( x) y^2$ is still a unit, we have that $f=\theta y^3 + u x^4 \looparrowright y^3 + \theta^{-1}ux^4$. Thus, by setting\footnote{The numbers in the exponents come from solving a set of linear equations whose goal is to rewrite all of the unit coefficients. We use this technique extensively throughout the next couple of Lemmata.} $x_1=\theta^{-\frac{1}{4}}u^{\frac{1}{4}}x$ we get that $f \looparrowright y^3 +x_1^4$.
        \item \underline{If $b(x)$ is a unit} then we have that $f = \theta y^3 + byx^3 + ux^4$ for $u=c(x)+ a_1( x)y^2$. Thus, by choosing $x_1$ and $y_1$ such that $x= \theta b^{-3}u^2 x_1$ and $y=  \theta b^{-4} u^3 y_1$ we get that 
        \begin{equation*}
            f=\theta^4 b^{-12} u^9 (y_1^3 + y_1x_1^3 + x_1^4) \looparrowright y_1^3 + y_1x_1^3 + x_1^4.
        \end{equation*}
        
        \end{enumerate}
        \item \underline{If $\ord\left(a\right) = 1$} then there exists some unit $a_1(x)$ that is a power series in $x$ such that $a(x)=xa_1(x)$. Thus we get that 
        \begin{equation*}
         f =  \theta  y^3 + (a_1( x) y + b( x)) y x^3 + c\left( x\right) x^4
         \end{equation*}
        Therefore:
        \begin{enumerate}
            \item \underline{If $b(x)$ is a unit} we get that $u=a_1( x) y + b( x)$ and thus $f=\theta  y^3 + u y x^3 + c\left( x\right) x^4$ which gives us that $f \looparrowright y_1^3 + y_1 x_1^3 + x_1^4$ for some $x_1$ and $y_1$ by a similar computation as in the previous case. 
            \item \underline{If $b(x)$ is not a unit} then there exists some $b_1(x)$ such that $b(x)=b_1(x)x$, and thus 
            \begin{equation*}
         f =  \theta  y^3 + a_1( x) y^2x^3 + (c\left( x\right)+b_1( x) y) x^4,
         \end{equation*}
         and since $c(x)$ is a unit we have that $u=(c\left( x\right)+b_1( x) y)$ is a unit as well. Therefore $f =  \theta  y^3 + a_1( x) y^2x^3 + u x^4$. By choosing $x_1$ such that $x=x_1-\frac{ay^2}{4u}$ we get that  
         \begin{equation*}
             f =  \theta  y^3 + a_1 y^2\left( x_1-\frac{a_1y^2}{4u} \right)^3 + u \left( x_1-\frac{a_1y^2}{4u} \right)^4, 
         \end{equation*}
         and thus we can find some new units $u_1$ and $u_2$ such that $f = u_1y^3 + u_2x_1^4$, which can be $f \looparrowright y_1^3 + x_2^4$ for some $y_1$ and $x_2$ by a similar computation as before. 
        \end{enumerate}
        
        \item \underline{If $\ord\left(a\right) = 0$} then:
        \begin{enumerate}
            \item \underline{If $b(x)$ is not a unit} then there exists some $b_1(x)$ such that $b(x)=b_1(x)x$ and thus 
            \begin{equation*}
                f =  \theta  y^3 + a\left( x\right) y^2  x^2 + (b_1( x) y + c\left( x\right) )x^4.
            \end{equation*}
            Since $c(x)$ is a unit we get that $u=b_1( x) y + c\left( x\right)$ and thus $f = \theta y^3 + a(x) x^2 y^2 +  ux^4 \looparrowright y_1^3 + x_1^2y_1^2 + x^4$ by a similar computation to as before. 
            \item \underline{If $b(x)$ is a unit} then as in the previous case, by choosing $x_1$ such that $x=x_1-\frac{b(x)y}{4c(x)}$, we get that \begin{equation*}
                f =  \theta  y^3 + a\left( x\right) y^2  \left(x_1-\frac{b(x)y}{4c(x)}\right)^2 + b\left( x\right) y \left(x_1-\frac{b(x)y}{4c(x)}\right)^3 + c\left( x\right) \left(x_1-\frac{b(x)y}{4c(x)}\right)^4.
            \end{equation*} 
            Thus, there are units $u_1, u_2, u_3$ such that $f=u_1y^3 + u_2 x_1^2y^2 + u_3x^4$, which can be written as $y_1^3 + x_2^2y_1^2 + x_2^4$ for some $x_2$ and $y_1$, as before. 
        \end{enumerate}
        
    \end{enumerate}

    Finally, we show that $y^3 +  y x^3 +  x^4 \looparrowright y^3 +  y^2 x^2 +  x^4$. By choosing $x_1$ such that $x=x_1-\frac{y}{4}$ then we get that $y^3 +  y x^3 +  x^4=y^3 +  y (x_1-\frac{y}{4})^3 +  (x_1-\frac{y}{4})^4$. By expanding the expression, we can find units $u_1, u_2,u_3$ such that $y^3 +  y x^3 +  x^4= u_1y^3 + u_2x_1^2 y^2 + u_3x_1^4$, which in turn can be written as $y_1 x_2^2y_1^2 + x_2^4$ for some $x_2$ and $y_1$, by a similar computation as before. 
\end{proof}

\begin{lemma}[$E_7$ case]
     In the settings of Lemma~\ref{lem:jet_3-form}, if $\ord\left(b\right)=0$ and $\ord\left(c\right)\geq 1$ then $f$ can be written as  either $ y^3 + y x^3$ or $y^3 +  y^2  x^2 +  y x^3$.
\end{lemma}

\begin{proof}
    We can write $c\left( x\right)= xc_1\left( x\right)$ for some $c_1$ that is a power series in $x$. Therefore,  we can write 
    \begin{equation*}
         f =  \theta  y^3+a\left( x\right) y^2 x^2 +  x^3\left( yb\left( x\right) + c_1\left( x\right) x^3\right).
    \end{equation*}
    By choosing $y_1$ such that  $ y_1  = yb\left( x\right)+c_1\left( x\right) x^3$ and noticing that $y=\frac{y_1-c_1(x)x^3}{b(x)}$,  we get that
    \begin{equation*}
        f = \frac{\theta  y_1^3}{b\left( x\right)^3}+\frac{a\left( x\right)}{b\left( x\right)} y_1^2 x^2 + u  x^3 y_1 
    \end{equation*}
    for some unit $u$, which in turn gives us $f \looparrowright y_1^3 + \alpha y_1^2  x^2 + \beta  y_1 x^3$ for some $\alpha$ and some unit $\beta$. Now, we have two cases:
    \begin{enumerate}
        \item \underline{If $\alpha$ is a unit} then by setting $x=\alpha^{0.5}x_1$ we can write $f$ as $y_1^3 +  y_1^2  x_1^2 + \beta_1  y_1 x_1^3$ for some unit $\beta_1$. Therefore, by setting $ y_1 = \beta_1^2 y_2$ and $ x_1 = \beta_1 x_2$ we get that $f \looparrowright y_2^3 +  y_2^2  x_2^2 +  y_2 x_2^3$.
        \item \underline{If $\alpha$ is not a unit} then we can find some $\alpha_1$ and $\alpha_2$ such that $\alpha = \alpha_1 x + \alpha_2y_1$. Thus we can conclude that 
        \begin{equation*}
            y_1^3 + (\alpha_1 x + \alpha_2y_1) y_1^2  x^2 + \beta  y_1 x^3= y_1^3 (1+\alpha_2x^2) + y_1x^3(\beta + y_1 \alpha_1).
        \end{equation*}
        Yet we have that $u_1 = (1+\alpha_2x^2)$ and $u_2 = \beta + y_1 \alpha_1$ are both units, and so $y_1^3 + \alpha y_1^2  x^2 + \beta  y_1 x^3= u_1 y_1^3 + u_2 y_1 x^3$. Thus, by setting $x_1=u_1^{-1}u_2^{-1}x$ and $y_2 = u_1^{-1}u_2^{-1}y_1$ we get that $f \looparrowright y_2^3 + y_2x_1^3$. 
    \end{enumerate}
\end{proof}

\begin{lemma}[$E_8$ case]
     In the settings of Lemma~\ref{lem:jet_3-form}:
     \begin{enumerate}
         \item If $\ord\left(b\right)\geq 2$ and $\ord\left(c\right)=1$ then $f$ is can be written as either $ y^3 +  y^2  x^2 +  x^5$, $ y^3 +  y^2  x^3 +  x^5$, $y^3 +  y^2  x^4 + x^5$, or $ y^3 +  x^5$.
         \item If $\ord\left(b\right) = 1$ and $\ord\left(c\right)=1$ then $f$ can be written as either $   y^3 +  y^2  x^2 +  y x^4 +  x^5$, $ y^3 +  y^2  x^3 +  y x^4 +  x^5$, or $ y^3 +  x^5 +   y x^4$.
     \end{enumerate}
\end{lemma}

\begin{proof}
    \begin{enumerate}
      \item  There exists some unit $c_1\left(x\right)$ and some $b_1\left(x\right)$ such that $ f =   \theta  y^3 + a\left( x\right) y^2  x^2 +  x^5\left( b\left( x\right) y + c_1\left( x\right)\right)$. Since $c_1\left( x\right)$ is a unit then so is $u_1=b\left( x\right) y + c_1\left( x\right)$. Now, we have the following cases:
    \begin{enumerate}
        \item \underline{If $\ord(a)=0$} then $  f =  \theta  y^3 + a\left( x\right)  y^2  x^2 + u_1 x^5$, which would give us $f \looparrowright y_1^3+ y_1^2 x_1^2 +  x_1^5$ by setting $ y = y_1\frac{u_1^2\theta^3}{a\left( x\right)^5}$ and $ x = x_1\frac{u_1 \theta^2}{a\left( x\right)^3}$. 
        \item \underline{If $\ord(a)=1$} then there exists some unit $a_1(x)$ such that $a(x)=a_1(x)x$, and so we have that $  f=  \theta  y^3 + a_1\left( x\right) y^2  x^3 + u_1 x^5$, which would give us $   f \looparrowright y_1^3 +  y_1^2  x_1^3 +  x_1^5$ under the renaming $ y = y_1\frac{a_1\left( x\right)}{\theta} \left( \frac{\theta^2 u_1}{a_1\left( x\right)^3} \right)^{\frac{3}{4}}$ and $ x = x_1\left( \frac{\theta^2 u_1}{a_1\left( x\right)^3} \right)^{\frac{1}{4}}$. 
        
        \item \underline{If $\ord(a)=2$} then there exists some unit $a_1(x)$ such that $a(x)=a_1(x)x^2$, and so $  f =  \theta  y^3 + a_1\left( x\right) y^2  x^4 + u_1 x^5$, and which in turn would give us $  f \looparrowright  y_1^3 +  y_1^2  x_1^4 + x_1^5$ by a similar computation as before. 
        
        \item \underline{If $\ord(a) \geq 3$} then there exists some $a_1(x)$ such that $a(x)=a_1(x)x^3$, and thus $u_2 = u_1+a_1\left( x\right) y$ is a unit with $  f=  \theta  y^3 + u_2 x^5$, and therefore by setting $ y = \theta^3 u_2^2  y_1$ and $ x = \theta^2 u_2  x_1$ we get that $f \looparrowright y_1^3 +  x_1^5$. 
    \end{enumerate}
    
    \item  There exists some units $c_1\left(x\right)$ and  $b_1\left(x\right)$ such that $b(x)=b_1(x)x$ and $c(x)=c_1(x)x$, and so $ f =   \theta  y^3 + a_1\left( x\right) y^2  x^2 + b_1\left( x\right) y x^4 + c_1\left( x\right) x^5$.  Thus, we have three cases to check:
    \begin{enumerate}
        \item \underline{If $\ord(a)=0$} then  $ f \looparrowright   \theta_1  y^3 + a(x)  y^2  x^2 +  b_1(x) y x^4 + c_1(x) x^5 $ which would give us $f \looparrowright  y^3 +  y^2  x^2 +  y x^4 +  x^5$. 
        \item \underline{If $\ord(a)=1$} then there exists some unit $a_1(x)$ such that $a(x)=a_1(x)x$, and so $ f \looparrowright  \theta_1  y^3 + a_1\left( x\right)  y^2  x^3 + b_1(x) y x^4 + c_1(x) x^5$ which would give us $  f \looparrowright   y^3 +  y^2  x^3 +  y x^4 +  x^5$. 
        \item \underline{If $\ord(a) \geq 2$} then there exists some $a_1(x)$ such that $a(x)=a_1(x)x^2$ and so we have that $b_1(x)+ ya_1\left( x\right)$ is a unit, and therefore $   \theta_1  y^3 + a_1( x) y^2  x^{2} + b_1(x) y x^4 + c_1(x) x^5 =   \theta_1  y^3 + c_1(x) x^5 + u_1  y x^4$, which would give us $ f \looparrowright  y^3 +  x^5 +   y x^4$.
    \end{enumerate}
    \end{enumerate}
\end{proof}

We now turn to showing the specific equivalences based on the value of $\chara(\kappa)$. Recall that in Proposition~\ref{prop:splitting_n=2}, we assume in addition that $R$ is Henselian and that $\kappa$ is algebraically closed (this assumption is important as it allows us to take roots of units in $R$ of degree different than $\chara(\kappa)$). 

\begin{lemma}[The $E_6$ case when $\chara(\kappa)\neq 3$]
    If $\chara(\kappa) \neq 3$ then  $  y^3+  x^4 +  y^2 x^2 \looparrowright    y^3+  x^4$. 
\end{lemma}

\begin{proof}
    We can write 
    \begin{equation*}
          y^3+  x^4 +  y^2 x^2 =   \left( y+\frac{ x^2}{3}\right)^3 +  x^4\left(1-\frac{y}{3}-\frac{ x^2}{27}\right),
    \end{equation*}
    which tells us that if we set $y_1 = y+\frac{x^2}{3}$, there exists some unit $u$ such that $ y^3+  x^4 +  y^2 x^2= y_1^3+u x^4=u(u^{-1} \cdot y_1^3+x^4)$. But, since $p>3$, we can take the $3$rd root of $u^{-1}$,  and so by setting $y_2=y_1 u^{-\frac{1}{3}}$ we have that $u^{-1} \cdot y_1^3+x^4= (u^{-\frac{1}{3}} y_1)^3+ x^4 = y_2^3+ x^4$, as desired.
\end{proof}

\begin{lemma}[The $E_7$ case when $\chara(\kappa)\neq 3$]
    If $\chara(\kappa) \neq 3$ then $   y^3 + y x^3 +  y^2 x^2 \looparrowright    y^3 + y x^3$. 
\end{lemma}

\begin{proof}
    We can write $ y x^3 +  y^2 x^2= y\left( x^3+ y x^2\right)$. Therefore, since $\chara(\kappa) \neq 3$ then $3 \in R$ is invertible and so by completing the cube, there exists some unit $u$ such that 
    \begin{equation*}
        y^3 + y x^3 +  y^2 x^2=   y\left( x+\frac{ y}{3}\right)^3 + u y^3.
    \end{equation*}
    Thus by setting $x_1=x+\frac{ y}{3}$ we get that $y^3 + y x^3 +  y^2 x^2=u_1(y\cdot u_1^{-1}\cdot x_1^3+y^3)$ for some unit $u_1$. By setting $x_2=u_1^{-\frac{1}{3}}x_1$ we can conclude that $y^3 + y x^3 +  y^2 x^2=u_1(y x_2^3+y^3)$, as desired. 
\end{proof}

\begin{lemma}[The $E_8$ case when $\chara(\kappa) \neq 3$]
    If $\chara(\kappa) \neq 3$ then we have the following equivalences:
    \begin{enumerate}
        \item $  y^3+ y x^4+ x^5+ y^2 x^2 \looparrowright   y^3+ y^2 x^2+ x^5$. 
        \item $  y^3+ y x^4+ x^5+ y^2 x^3 \looparrowright   y^3+ y^2 x^3+ x^5$.
        \item $   y^3 + x^5 +  y^2 x^2 \looparrowright    y^3 + x^5 +  y x^4$.
        \item $   y^3 + x^5 +  y^2 x^3 \looparrowright    y^3 + x^5$.
        \item $  y^3+ y^2 x^4+ x^5 \looparrowright   y^3+ x^5$.
    \end{enumerate}
\end{lemma}

\begin{proof}
    \begin{enumerate}
        \item By completing the cube we have that $  y^3+ y x^4+ x^5+ y^2 x^2= \left( y+\frac{ x^2}{3}\right)^3 +  x^5+u_1 y x^4$ for some unit $u_1$. Therefore, by setting $y_1 = y-\frac{x^2}{3}$ we get $\left( y+\frac{ x^2}{3}\right)^3 +  x^5+u_1 y x^4=  y_1^3+ x^5+u_1\left( y_1-\frac{ x^2}{3}\right) x^4$. Now, by setting $y_2=y_1-\frac{x^2}{3}$ we get that $y_1^3+ x^5+u_1\left( y_1-\frac{ x^2}{3}\right) x^4=  y_2^3+u_2 x^5+u_3 y_2 x^4$ for some units $u_2$ and $u_3$. Thus, by setting $x_1=u_2^{\frac{1}{4}}x$ we have that  $ y_2^3+u_2 x^5+u_3 y_2 x^4 = y_2^3+u_4 x_1^5+ y_2 x_1^4$ for some $u_4$, and thus  $y_2^3+u_4 x_1^5+ y_2 x_1^4 \looparrowright y^3+ x^5+ y x^4$. 
        \item By completing the cube, there exists a unit $u_1$ such that $  y^3+ y x^4+ x^5+ y^2 x^3= \left( y+\frac{ x^3}{3}\right)^3+ y x^4+u_1 x^5$. Thus, if we set  $y_1=  y-\frac{x^3}{3}$ we get $\left( y+\frac{ x^3}{3}\right)^3+ y x^4+u_1 x^5=y_1^3+\left( y_1-\frac{ x^3}{3}\right) x^4+u_1 x^5$. In turn, we can write $y_1^3+\left( y_1-\frac{ x^3}{3}\right)^3 x^4+u_1 x^5= u_2 y_1^3+ u_3x^5+ y_1^2 x^3$ for some units $u_2$ and $u_3$,  and we can see that $ u_2 y_1^3+ u_3x^5+ y_1^2 x^3 \looparrowright  y^3+ x^5+ y^2 x^3$. 
        \item  If we complete the cube and define $y_1= y-\frac{x^2}{3}$ we can write 
        \begin{equation*}
               y^3 +  x^5 +  y^2 x^2 =  \left(y+\frac{x^2}{3}\right)^3+ x^5-\frac{yx^4}{9} - \frac{x^6}{27}=y_1^3 +  x^5 -\frac{ y_1 x^4}{9} - \frac{2x^6}{27}. 
        \end{equation*}

          Now, we can write  $y_1^3 +  x^5 -\frac{ y_1 x^4}{9} - \frac{2x^6}{27} =  y_1^3 + u_1  x^5 + u_2 y_1 x^4$ for some units $u_1$ and $u_2$. Thus, if we define $y_2 = u_2 y_1$ we get $y_1^3 + u_1  x^5 + u_2 y_1 x^4=y_2^3 + u  x^5 +  y_2 x^4$ for some $u$, and  by setting $y_3 = y_2u^{\frac{4}{2}}$ and $x_1= xu^{\frac{3}{2}}$ we get that $y_2^3 + u  x^5 +  y_2 x^4 \looparrowright y^3 +  x^5 +  y x^4$.

        \item If we complete the cube and define $y_1 = y-\frac{x^2}{3}$, then there exists some unit $u$ such that
        \begin{equation*}
               y^3 + x^5 +  y^2 x^3 =   \left( y+\frac{ x^3}{3}\right)^3 + x^5 +  y^2 x^3 - \frac{ y x^6}{3} - \frac{x^9}{27}=u((u^{-\frac{1}{3}}y_1)^3 + x^5),
        \end{equation*}
        which in turn can be written as $y_2^3+x^5$ for $y_2=u^{-\frac{1}{3}}y_1$. 
    \item  By completing the cubic, there exists some unit $u$ such that $  y^3+ y^2 x^4+ x^5= \left( y+\frac{ x^2}{3}\right)^3+u x^5$. By defining $y_1 = y+\frac{ x}{3}$  we get $\left( y+\frac{ x}{3}\right)^3+u x^5=y_1^3+ux_1^5$, and so by setting $y_2=u^{-\frac{1}{3}}y_1$ we get $y_1^3+ux_1^5 = u(y_2^3+x_1^5) \looparrowright y_2^3  +x_1^5$. 
    \end{enumerate}
\end{proof}

\begin{lemma}[The $E_8$ case when $\chara(\kappa) \neq 5$]
    If $\chara(\kappa) \neq 5$ then we have the following equivalences:
    \begin{enumerate}
        \item $y^3+ y^2 x^4+ x^5 \looparrowright   y^3+ x^5$.
        \item $y^3+ y x^4+ x^5+ y^2 x^2 \looparrowright   y^3+ y^2 x^2+ x^5$.
        \item $y^3+ y x^4+ x^5+ y^2 x^3 \looparrowright   y^3+ y x^4+ x^5$.
        \item $y^3 + x^5+  y x^4 \looparrowright    y^3 + x^5+  y^2 x^3$.
        \item $   y^3 + x^5+  y x^4 \looparrowright    y^3 + x^5$.
    \end{enumerate}
\end{lemma}

\begin{proof}
    \begin{enumerate}
        \item If we complete the quintic, there exists some unit $u$ such that $  y^3+ y^2 x^4+ x^5 = u y^3+\left( x+\frac{ y^2}{5}\right)^5$. Therefore, by defining $x_1 = x-\frac{ y^2}{5}$ and $y_1=u^{\frac{1}{3}}y$ we get $ y^3+ y^2 x^4+ x^5 = y_1^3 + x_1^5$. 
        \item If we complete the quintic then there exists some units $u_1$ and $u_2$ such that $  y^3+ y x^4+ x^5+ y^2 x^2= u_1 y^3+u_2 y^2 x^2+\left( x+\frac{ y}{5}\right)^5$, and by defining $x_1=x-\frac{ y}{5}$ we get $y^3+ y x^4+ x^5+ y^2 x^2= u_3 y^3+u_4 y^2\left( x_1-\frac{ y}{5}\right)^2+ x_1^5$ for some units $u_3$ and $u_4$. In turn, we have that $u_3 y^3+u_4 y^2\left( x_1-\frac{ y}{5}\right)^2+ x_1^5= u_5 y^3+u_6 y^2 x_1^2+ x_1^5$ for some units $u_5$ and $u_6$. In turn, since $u_5$ is a unit we get $u_5 y^3+u_6 y^2 x_1^2+ x_1^5 \looparrowright y^3+ u_7 y^2 x_1^2+ u_8x_1^5$ for some units $u_7$ and $u_8$. Thus, by setting $x_2=u_7^{0.5}x_1$ we get that $y^3+ u_7 y^2 x_1^2+ u_8x_1^5 = y^3+ y^2 x_2^2+ u_9x_2^5$ for some unit $u_9$, and by setting $x_2=u_9x_3$ and $y_1=u_9^2y$ we have that $y^3+ y^2 x_2^2+ u_9x_2^5= u_9^6(y_1^3+ y_1^2 x_3^2+ x_3^5) \looparrowright y^3+ y^2 x^2+ x^5$. 
        \item  By completing the quintic we have that $  y^3+ y x^4+ x^5+ y^2 x^3= u_1 y^3+\left( x+\frac{ y}{5}\right)^5- y^2 x^3$ for some unit $u_1$. Therefore, by defining $x_1 = x-\frac{y}{5}$ we get that $u_1 y^3+\left( x+\frac{ y}{5}\right)^5- y^2 x^3= u_2  y^3+ x^5- y\left( x-\frac{ y}{5}\right)^5$ for some unit $u_2$. Thus, there exists some unit $u_3$ such that 
        \begin{equation*}
            u_2  y^3+ x^5- y\left( x-\frac{ y}{5}\right)^5=  u_3  y^3+ x^5 +  y^2 x^2 \looparrowright y^3+ x^5+ y^2 x^2.
        \end{equation*}
        \item We can write 
    \begin{equation*}
           y^3 + x^5+  y x^4 =   \left( x+\frac{y}{5}\right)^5-2 y^2 x^3 + u y^3
    \end{equation*}
    for some unit $u$. Therefore by setting $ x_1 = x-\frac{ y}{5}$ we get that $   y^3 + x^5+  y x^4 = x_1^5 - 2 y^2 (x_1-\frac{y}{5})^3 + u_0 y^3$ for some other unit $u_0$. We can write $x_1^5 - 2 y^2 (x_1-\frac{y}{5})^3 + u_0 y^3$ as $x_1^5-2 x_1^3 y^2+u_1 y^3$ for some unit $u_1$, and we have that 
    \begin{equation*}
        x_1^5-2 x_1^3 y^2+u_1 y^3 =u_1 (u_1^{-1} x_1^5+ y^3 + (-2u_1^{-1})x_1^3y^2) \looparrowright y^3 + x^5+  y^2 x^3. 
    \end{equation*}
    \item By completing the quintic, there exists some unit $u$ such that  $y^3+x^5+yx^4=uy^3 + (x+\frac{y}{5})^5 - \frac{2}{5}x^2y^2(x+\frac{y}{5})$. Thus, by setting $x_1=x+\frac{y}{5}$ we get that $y^3+x^5+yx^4=uy^3+x_1^5-\frac{2}{5}(x_1-\frac{y}{5})^2y^2x_1$. Thus, there are units $u_1$ and $u_2$ such that $y^3+x^5+yx^4= u_1y^3 + x_1^5 + u_2x_1^3y^3$. Therefore, by finding $x_2$ and $y_1$ such that $y=u_1^{\frac{1}{2}}u_2^{-\frac{3}{20}}y_1$ and $x_1=u_1^{\frac{1}{2}}u_2^{-\frac{1}{4}}x_2$, we get that $u_1y^3 + x_1^5 + u_2x_1^2y^3 \looparrowright y_1^3 + x_2^5 + x_2^2y_1^3$. 
    \end{enumerate}
\end{proof}

\begin{remark}
    \textup{In the settings of Lemma~\ref{lem:jet_3-form}, if $\ord\left(b\left( x\right)\right) \geq 1$ and $\ord\left(c\left( x\right)\right) \geq 2$ then there exists some $b_1(x)$ and $c_1(x)$ such that $b(x)=b_1(x)x$ and $c(x)=c_1(x)x^2$, and thus $\theta y^3 + a\left( x\right) y^2  x^2 + b\left( x\right) y x^3 + c\left( x\right) x^4 = \theta y^3 + a\left( x\right) y^2  x^2 + b_1\left( x\right) y x^4 + c_1\left( x\right) x^6 \in \langle  y,  x^2 \rangle^3$.  }
\end{remark}

We can summarize this sections results in the following corollary:

\begin{corollary}\label{cor:final}
    Let $(R, \mathfrak{m}, \kappa)$ be a  $2-$dimensional complete regular local ring and let $f \in R$ be of order $3$. Then either there exists some $x,y \in \mathfrak{m}$ such that $f \in \langle x, y^2 \rangle^3$ or $f$ can be written as one of the elements in the~\nameref{thm:mainintro} (with $n=2$). 
\end{corollary}

\begin{remark}\label{rem:final}
    \textup{Note if $(R, \mathfrak{m}, \kappa)$ is $2-$dimensional complete regular local ring and $f \in \mathfrak{m}^2$, then by Proposition 4.1 and Proposition 4.2 in~\cite{buchweitz1987cohen} we know that if $f$ can be written as either $A_\infty$ or $D_\infty$ then $f$ has countable Cohen-Macaulay type (assuming $\kappa$ is uncountable). In addition, if $f$ can be written as one of the $ADE$ singularities (that are neither $A_\infty$ nor $D_\infty$) then $f$ must have finite Cohen-Macaulay type, by a similar computation to the one done in Section 9.4 of~\cite{leuschke2012cohen}. Alternatively, we can see that they satisfy the Drozd-Ro\u{ı}ter conditions for one dimensional rings, as first stated in~\cite{drozd1967commutative} (for more information see Chapter 4 in~\cite{leuschke2012cohen}). In fact,  Green and Reiner in~\cite{green1978integral}, together with the work of \c{C}imen in~\cite{ccimen1998one}, tells us that if $(R, \mathfrak{m})$ is a one-dimensional local ring that is analytically unramified (i.e. the $\mathfrak{m}-$adic completion of $R$ is reduced) with sparse Cohen-Macaulay type, then $f$ must have finite Cohen-Macaulay type (for more information see Chapter 4 of ~\cite{leuschke2012cohen}).}
\end{remark}

Proposition~\ref{prop:n-2}, Proposition~\ref{prop:n-2}, and Corollary~\ref{cor:final} gives us the proof of the~\nameref{thm:mainintro}. 

\section{Generalized Double Branch Cover over Complete Regular Rings}

The goal of this section is to provide a generalization of the results of Chapter 8.2  in~\cite{leuschke2012cohen} to complete regular rings (without looking at power series), which can be thought of as a partial generalization of Kn\"{o}rrer's theorem, first presented in~\cite{knorrer1986cohen}. Specifically, let $(R, \mathfrak{m}, k)$ be a complete regular local ring with $\text{char}(k) \neq 2$ and let $b, a_1, \dots, a_n \in R$ be a minimal sequence that generates $\mathfrak{m}$. Assume that $g \in R$ is a power series in $a_1, \dots, a_n$. Denote $\overline{R} = \frac{R}{\langle b \rangle}$ and denote $\overline{g} = g \mod \langle b \rangle$. The the goal of this section is to understand the relationship between the Cohen-Macaulay type of $S^{\flat} = \frac{\overline{R}}{\langle \overline{g} \rangle}=\frac{R}{\langle b,g \rangle}$ and of $S=\frac{R}{\langle b^2 + g \rangle}$.

\begin{remark}\label{rem:double_cover_name}
    \textup{We call $S$ the "generalized double cover" of $S^\flat$ as an analogue of $\frac{R[[z]]}{\langle h+z^2 \rangle}$ being the double cover of $\frac{R}{\langle h \rangle}$ for $h \in \mathfrak{m}^2$. For more information on double covers and Cohen-Macaulay type, see Chapter 12 in~\cite{yoshino1990maximal}. }
\end{remark}

\begin{lemma}\label{lem:b_not_zd}
    $b$ is not a zero divisor of $S$.
\end{lemma} 

\begin{proof}
    Assume that $h \in R$ satisfies $bh \in \langle g+b^2 \rangle$ and we show that $h \in \langle g+b^2 \rangle$. Since  $bh \in \langle g+b^2 \rangle$ then there exists some $r \in R$ such that $bh=r(g+b^2)$. Therefore $b(h-rb)=gr$. Since $R$ is a regular local ring then it is a UFD, and since $b, a_1, \dots, a_n$ form a regular sequence with $g$ being a power series in $a_1, \dots, a_n$, then we can conclude that $r$ must be divisible by $b$, that is, there exists some $r_0 \in R$ such that $r = b r_0$. This gives us that $h-b^2r_0 = gr_0$, and so $h = r_0(b^2+g) \in  \langle g+b^2 \rangle$, as desired. 
\end{proof}

\begin{definition}
    \begin{itemize}
        \item Given $N \in \mathcal{MCM}(S)$, we denote $N^\flat = \frac{N}{bN}$, viewed as a module over $S^\flat$. 
        \item Given $M \in \mathcal{MCM}(S^\flat)$, we denote $M^\sharp = \syz^{S}_1(M)$, that is, the first syzygy of $M$ with respect to a minimal resolution over $S$, which we view as a module over $S$. 
    \end{itemize}
\end{definition}

The maps $N \mapsto N^\flat$ and $M \mapsto M^\sharp$ allow us to understand the relationship between $\mathcal{MCM}(S)$ and $\mathcal{MCM}(S^\flat)$. As we shall see, these are not direct maps between the two sets, but their decompositions into irreducible components are for special MCM modules.

\begin{remark}
    \textup{ As  in Definition 8.12 in~\cite{leuschke2012cohen} and the discussion after it, note that the module $M^\sharp$ is well defined (up to isomorphism). This is true since by Lemma~\ref{lem:b_not_zd}, $b$ is not a zero-divisor, and thus we have a short exact sequence }
    \begin{equation*}
    \begin{tikzcd}
    0\arrow[r] & R \arrow[r, "z \mapsto bz"] & R \arrow[r] & \overline{R} \arrow[r] & 0.
    \end{tikzcd}
    \end{equation*}
    \textup{Therefore we can conclude that $R$ is indeed the first syzygy of $\overline{R}$ as an $R-$module.} 
\end{remark}

We recall a result by Eisenbud from~\cite{eisenbud1980homological} that we use extensively throughout this appendix. For more details, see Chapter 7 of~\cite{yoshino1990maximal}  or Section 1 of~\cite{buchweitz1987cohen}.\\

\begin{remark}
    \textup{Given a ring $A$, we denote by $\textup{Mat}_{a,b}(A)$ the set of $a \times b$ matrices over $A$. Given a module $M$ over $A$, we denote by $\mathbb{1}_M$ the identity map on $M$. Specifically, we denote $\mathbb{1}_{a \times b} = \mathbb{1}_{\textup{Mat}_{a,b}(A)}$, which we view as the $(a \times b)-$unit matrix over $A$. If the module $M$ is clear, we simply denote $\mathbb{1}$.  }
\end{remark}

\begin{definition}
    Let $(A,\mathfrak{n})$ be a be a regular local ring and let $r \in \mathfrak{n}$. A \textbf{reduced matrix factorization} of $r$ is a pair of matrices $\left(\varphi, \psi\right) \in \textup{Mat}_{a,b}(A) \times \textup{Mat}_{b,a}(A)$  such that:
    \begin{enumerate}
        \item The image of $\varphi$ lies in $\mathfrak{n} \cdot  A^b$,
        \item The image of $\psi$ lies in $\mathfrak{n} \cdot A^a$, 
        \item $\varphi\psi = r \cdot \mathbb{1}_{a\times a}$,
        \item $ \psi \varphi = r \cdot \mathbb{1}_{b \times b} $
    \end{enumerate}
    we say that two matrix factorizations $(\varphi_1, \psi_1)$ and $(\varphi_2, \psi_2)$ are equivalent if there exists some invertible matrices $\alpha$ and $\beta$ such that $\beta \varphi_1 = \varphi_2 \alpha$ and $\alpha \psi_1 = \psi_2 \beta$
\end{definition}

\begin{theorem}[Eisenbud~\cite{eisenbud1980homological}]\label{thm:matrix_fact}
    Let $(A,m)$ be a be a regular local ring and let $r \in A$ be a non unit. Then there there exists a bijection between $\mathcal{MCM}(\frac{A}{\langle r \rangle})$ and equivalence classes of reduced matrix factorizations of $r$. This bijection is given by $(\varphi, \psi) \mapsto \coker(\varphi)$. 
\end{theorem}

\begin{remark}\label{rem:matrix}
\begin{enumerate}
    \item \textup{Given a reduced matrix factorization $(\varphi, \psi)$, then we denote $\coker(\varphi, \psi) = \coker(\varphi)$ for the sake of consistency, as denoted in Notation 8.4 in~\cite{leuschke2012cohen}.}
    \item \textup{By Proposition 7.7 in~\cite{yoshino1990maximal}, if $M=\coker(\varphi, \psi)$ then $\syz_R^1(M)=\coker(\psi, \varphi)$. }
    \item \textup{The proof of Proposition~\ref{prop:ideal} relies heavily on Theorem~\ref{thm:matrix_fact}. Specifically, given some $[M] \in \mathcal{MCM}\left( \frac{A}{\langle r \rangle} \right)$ with some matrix factorization $(\varphi, \psi)$, we can define $I(\varphi, \psi)$ to be the ideal generated by all of the entries of $\varphi$ and $\psi$. Then we can show that$I(\varphi, \psi) \in \mathcal{IS}(r)$ and that $(\varphi, \psi) \mapsto I(\varphi, \psi)$ is surjective. For more information, see for example, in Theorem 9.2 in~\cite{leuschke2012cohen}.   }
\end{enumerate}
\end{remark}

Therefore, our goal is to compare $\mathcal{MCM}(S)$ and $\mathcal{MCM}(S^\flat)$ by comparing their corresponding matrix factorizations.

\begin{definition}
    Let $(\varphi, \psi)$ be a reduced matrix factorization of $\overline{g} \in \overline{R}$. Then we say that $(\varphi, \psi)$ is \textbf{liftable} (with respect to $g$) if there exists a reduced matrix factorization $(\varphi_1, \psi_1)$ of $g \in R$ such that $\overline{\varphi_1}=\varphi$, $\overline{\psi_1}=\psi$. We denote by $\mathcal{MCM}_l(S^\flat)$ the set of equivalence classes $[M] \in \mathcal{MCM}(S^\flat)$ such that $M \cong \coker(\varphi, \psi)$ where $(\varphi, \psi)$ is liftable.
   \end{definition} 

\begin{remark}
    \textup{Note that in general, we do not know if every module in $\mathcal{MCM}(S^b)$ has a matrix factorization that is liftable. In the power series case (as in Remark~\ref{rem:double_cover_name}), this is clear since we had an embedding $\frac{R[[z]]}{\langle z \rangle} \cong R \subset R[[z]]$, and so every matrix factorization over $R$ lifts to $R[[z]]$. Yet, in our generalized double cover case, we need not have an embedding $\frac{R}{\langle b \rangle} \subset R$.  }
\end{remark}

The proof of the following lemma is similar to the proof of Lemma 8.14 in~\cite{leuschke2012cohen}.

\begin{lemma}\label{lem_App:M_matrices}
    Let $[M]\in \mathcal{MCM}_l(S^\flat)$ with $ M =\coker(\varphi, \psi) $. Then 
    \begin{equation*}
        M^\sharp = \coker \left(
        \begin{bmatrix}
            \psi_1 & -b\mathbb{1}\\
            b\mathbb{1} & \varphi_1
        \end{bmatrix}
        ,
        \begin{bmatrix}
            \varphi_1 & b\mathbb{1}\\
            -b\mathbb{1} & \psi_1
        \end{bmatrix}
        \right)
    \end{equation*}
    with $\syz^{S}_1(M^\sharp)=M^\sharp$.
\end{lemma}

\begin{proof}
    Observe that they indeed form a reduced matrix factorization, and observe the following computation:
    \begin{equation*}
        \begin{bmatrix}
            0 & 1\\
            1 & 0
        \end{bmatrix}
        \begin{bmatrix}
            \varphi_1 & b\mathbb{1}\\
            -b\mathbb{1} & \psi_1
        \end{bmatrix}
        \begin{bmatrix}
            0 & 1\\
            1 & 0
        \end{bmatrix}
        =
        \begin{bmatrix}
            \psi_1 & -b\mathbb{1}\\
            b\mathbb{1} & \varphi_1
        \end{bmatrix}
        .
    \end{equation*}
    Now, it is enough to prove that if $\varphi_1 \in \Mat_{a \times b}(R)$ and $\psi_1 \in \Mat_{b \times a}(R)$ then we have an exact sequence
    \begin{equation*}
        R^{a}\oplus R^b \xrightarrow[]{
        \begin{bmatrix}
             \psi_1 & - b\mathbb{1}_b \\
             b\mathbb{1}_a & \varphi_1
        \end{bmatrix}
        }
        R^{b}\oplus R^a
        \xrightarrow[]{
        \begin{bmatrix}
            \varphi_1 & b \mathbb{1}_a
        \end{bmatrix}
        }
        R^a \xrightarrow[]{\pi} M \to 0,
    \end{equation*}
    where $\pi \colon R^a \to \overline{R}^a \to M$ is the composition of the projection with the cokernel map of $M$ (i.e. with $\varphi$ where $M=\coker(\varphi, \psi)$). Note that by the definition of $\pi$ it is clear that it is surjective and that its kernel is the image of  $\begin{bmatrix}
            \varphi_1 & b \mathbb{1}_a
        \end{bmatrix}$. Therefore it is enough to prove exactness at $R^b \oplus R^a$.  Since by Lemma~\ref{lem:b_not_zd} we have that $b$ is a non-zero divisor of $S$, then the columns of the following commutative diagram are exact: 

    \[\begin{tikzcd}
	&& 0 && 0 && 0 \\
	\\
	&& {R^a} && {R^b} && {R^a} \\
	\\
	&& {R^a} && {R^b} && {R^a} \\
	\\
	\cdots && {\overline{R}^a} && {\overline{R}^b} && {\overline{R}^a} && M && 0 \\
	\\
	&& 0 && 0 && 0
	\arrow[from=1-3, to=3-3]
	\arrow[from=1-5, to=3-5]
	\arrow[from=1-7, to=3-7]
	\arrow["{\psi_1}", from=3-3, to=3-5]
	\arrow["{\cdot b}", from=3-3, to=5-3]
	\arrow["{\varphi_1}", from=3-5, to=3-7]
	\arrow["{\cdot b}", from=3-5, to=5-5]
	\arrow["{\cdot b}", from=3-7, to=5-7]
	\arrow["{\psi_1}", from=5-3, to=5-5]
	\arrow[from=5-3, to=7-3]
	\arrow["{\varphi_1}", from=5-5, to=5-7]
	\arrow[from=5-5, to=7-5]
	\arrow[from=5-7, to=7-7]
	\arrow["\pi", dotted, from=5-7, to=7-9]
	\arrow["\varphi", from=7-1, to=7-3]
	\arrow["\psi", from=7-3, to=7-5]
	\arrow[from=7-3, to=9-3]
	\arrow["\varphi", from=7-5, to=7-7]
	\arrow[from=7-5, to=9-5]
	\arrow[from=7-7, to=7-9]
	\arrow[from=7-7, to=9-7]
	\arrow[from=7-9, to=7-11]
\end{tikzcd}\]

    \noindent In addition, observe that the bottom row is exact due to the fact that $(\varphi, \psi)$ is a matrix factorization. Note that since $\varphi \psi + b^2 \mathbb{1}_a = 0$ over $S$, then by diagram chasing we can conclude that $\ker(\pi)=\ima(\varphi_1+bR^a) = \ima(
    \begin{bmatrix}
        \varphi_1 & b \mathbb{1}_a
    \end{bmatrix}
    )$ and that 
    \begin{equation*}
        \ker(
    \begin{bmatrix}
        \varphi_1 & b \mathbb{1}_a
    \end{bmatrix}
        ) \supset \ima \left(
    \begin{bmatrix}
        \psi_1 & -b\mathbb{1}_b\\
        b\mathbb{1}_a  & \varphi_1
    \end{bmatrix}
        \right).
    \end{equation*}
        For the other inclusion, if $x$ and $y$ form a tuple in  are $\ker(
    \begin{bmatrix}
        \varphi_1 & b \mathbb{1}_a
    \end{bmatrix}
        )$, then we have that $\varphi_1(x)=-by$. Therefore by diagram chasing we can find some $x_1$ and $x_2$ such that 
        \begin{equation*}
            \begin{bmatrix}
                \psi_1 &
                -b\mathbb{1}_a
            \end{bmatrix}
            \begin{bmatrix}
                x_1\\
                x_2
            \end{bmatrix}
            =x,
        \end{equation*}
        which gives us that $\psi_1(x_1)=bx_2+x$. Therefore we can conclude that 
        \begin{equation*}
            b(bx_1 +\varphi_1(x_2))=-\varphi_1 \psi_1(x_1)+b\varphi_1(x_2)=-\varphi_1(\psi_1(x_1)-bx_2)=-\varphi_1(x)=by.
        \end{equation*}
        Since $b$ is a non-zero divisor, then we can conclude that $bx_1 +\varphi_1(x_2)=y$, which gives us that 
        \begin{equation*}
            \begin{bmatrix}
                b\mathbb{1}_a & \varphi_1
            \end{bmatrix}
            \begin{bmatrix}
                x_1\\
                x_2
            \end{bmatrix}
            =y,
        \end{equation*}
        and so the result follows. 
\end{proof}

The proof of the following lemma is similar to the proof of proposition 8.15 in~\cite{leuschke2012cohen}.

\begin{proposition}\label{prop_App:M_flatsharp}
    Given some $[M] \in \mathcal{MCM}_l(S^\flat)$, we have that 
    \begin{equation*}
        (M^\sharp)^\flat \cong M \oplus \syz^{S^\flat}_1(M).
    \end{equation*}
\end{proposition}

\begin{proof}
    If we denote by $(\Phi, \Psi)$ the matrix factorization of $M^\sharp$ as in Lemma~\ref{lem_App:M_matrices} we get that $(\Phi \otimes_{S} S^\flat, \Psi \otimes_{S} S^\flat)$ is a matrix factorization of $(M^\sharp)^\flat$. Yet, tensoring by $S^\flat$ kills the element $b$ in these matrices and so we can conclude that 
    \begin{equation*}
        (M^\sharp)^\flat = \coker 
        \left(
        \begin{bmatrix}
            \psi, & 0\\
            0 & \varphi
        \end{bmatrix}
        ,
        \begin{bmatrix}
            \varphi, & 0\\
            0 & \psi
        \end{bmatrix}
        \right),
    \end{equation*}
    and the result follows from the second item of Remark~\ref{rem:matrix}. 
\end{proof}


\begin{definition}
    We say that $N \in \mathcal{MCM}(S)$ is \textbf{rootable} (with respect to $g$) if there exists some matrix $\varphi$ over $R$ such that $N \cong \coker(b\mathbb{1}-\varphi, b\mathbb{1}+\varphi)$ and $\varphi^2 = -g \mathbb{1}$. We denote by $\mathcal{MCM}_r(S)$ the set of $[N] \in \mathcal{MCM}(S)$ such that $N$ is rootable. 
\end{definition}

\begin{remark}
    \textup{In the power series case (as in Remark~\ref{rem:double_cover_name}) we have that every element if $\mathcal{MCM}(S)$ is rootable, as in this case $S$ is a free $\overline{R}-$module, and so any MCM module over $S$ is a finitely generated free $\overline{R}-$module (for a complete proof, see Lemma 8.17 in~\cite{leuschke2000ascent}). Yet, in the general case $S$ need not be a module over $\overline{R}$, as they might have different characteristics. Therefore it is unknown if in general every $MCM$ module over $S$ is rootable. }
\end{remark}

The proof of the following proposition is similar to the proof of Proposition 8.18 in~\cite{leuschke2012cohen}.

\begin{proposition}\label{prop_App:N_flatsharp}
    Given some $N \in \mathcal{MCM}_r(S)$, we have that 
    \begin{equation*}
        (N^\flat)^\sharp \cong N \oplus \syz^{S}_1(N).
    \end{equation*}
\end{proposition}

\begin{proof}
Since $N$ is rootable there exists some $\varphi$ such that $N \cong \coker(b\mathbb{1}-\varphi, b\mathbb{1}+\varphi)$ and and $\varphi^2 = -g \mathbb{1}$. Therefore we can conclude that  $(\varphi, -\varphi)$ is a matrix factorization of $g$ and thus $\coker(\varphi, -\varphi) = N^\flat$. So, by Lemma~\ref{lem_App:M_matrices}, we get that 
\begin{equation*}
    (N^\flat)^\sharp = \syz_1^S(N^\flat) = \coker\left(
    \begin{bmatrix}
        -\varphi & -b\mathbb{1}_N\\
        b\mathbb{1}_N & \varphi
    \end{bmatrix}
    , 
    \begin{bmatrix}
        \varphi & b\mathbb{1}_N\\
        -b\mathbb{1}_N & -\varphi
    \end{bmatrix}
    \right). 
\end{equation*}
In addition, note that since $2$ is invertible in $R$, then we have the following matrix congruence
\begin{equation*}
    \begin{bmatrix}
        b \mathbb{1}_N - \varphi & 0\\
        0 & b \mathbb{1}_N + \varphi
    \end{bmatrix}
    =
    \frac{1}{2}
    \begin{bmatrix}
        1 & 1\\
        -1 & 1
    \end{bmatrix}
    \begin{bmatrix}
        -\varphi & -b\mathbb{1}_N\\
        b\mathbb{1}_N & \varphi
    \end{bmatrix}
    \begin{bmatrix}
        1 & 1\\
        -1 & 1
    \end{bmatrix}
    .
\end{equation*}
Therefore, since $N \cong \coker(b\mathbb{1}-\varphi, b\mathbb{1}+\varphi)$, we can concldue that 
\begin{equation*}
    (N^\flat)^\sharp = \coker(b \mathbb{1}_N - \varphi, b \mathbb{1}_N + \varphi ) \oplus \coker(b \mathbb{1}_N + \varphi, b \mathbb{1}_N - \varphi ) \cong N \oplus \syz^{S}_1(N),
\end{equation*}
as desired.
\end{proof}

In order to conclude our main result, we need to apply the Krull-Remak-Schmidt Theorem (over Henselian rings), first proven by Remark in~\cite{remak1911zerlegung} and then extended by Krull in~\cite{krull1925verallgemeinerte} and Schmidt in~\cite{schmidt1929unendliche}. For more information, see Chapter 1.2 of~\cite{leuschke2012cohen}. (Note that by Corollary 1.9 in~\cite{leuschke2012cohen} we have that every complete local ring is Henselian).

\begin{proposition}[Krull-Remak-Schmidt Theorem]\label{prop:KRS}
    Let $A$ be a Henselian local ring. Let $M_1, \dots, M_r$ and $N_1, \dots, N_l$ be indecomposable finitely generated modules over $A$ such that $M_1 \oplus \cdots \oplus M_r \cong N_1 \oplus \cdots \oplus N_l$. Then $r=l$ and $M_i \cong N_i$ for every $i$ (after potentially renumbering). 
\end{proposition}

The proof of the following proposition is similar to the proof of Corollary 8.16 in~\cite{leuschke2012cohen}.
\begin{proposition}\label{prop_App:M_decom}
    Given $N \in \mathcal{MCM}_r(S)$ and $M \in \mathcal{MCM}_l(S^\flat)$. Then:
    \begin{enumerate}
        \item $M^\sharp$ is either irreducible or decomposes into a sum of two indecomposable elements in $\mathcal{MCM}(S)$. 
        \item There exists some $N_1 \in \mathcal{MCM}(R)$ such that $M$ is a direct summand of $N_1^\flat$.
        \item $N^\flat$ is either irreducible or decomposes into a sum of two elements in $\mathcal{MCM}(S^\flat)$. 
        \item There exists some $M_1 \in \mathcal{MCM}(R^\flat)$ such that $N$ is a direct summand of $M_1^\sharp$. 
    \end{enumerate}
\end{proposition}

\begin{proof}
    We prove the proposition for $M$, as the case for $N$ is exactly the same. For the first part, assume towards contradiction that $M^\sharp$ can be written as a direct sum of more than 2 non trivial modules over $S$. Then $(M^\sharp)^\flat$ can be too. Yet, by Proposition~\ref{prop_App:M_flatsharp} we must have that $(M^\sharp)^\flat$ is a sum of exactly 2 indecomposable non-free $S-$modules, which is a contradiction. \\

    For the second part, if we write $M^\sharp = N_1 \oplus \cdots \oplus N_r$, where each $N_i$ is an indecomposable $MCM$ module over $R$, then non of them are free by Proposition~\ref{prop_App:M_flatsharp}. Then $M \oplus \syz_1^{S^\flat}(M) = (M^\sharp)^\flat = N_1^\flat \oplus \cdots \oplus N_r^\flat$. Thus, by Proposition~\ref{prop:KRS} we can conclude that $M$ must be isomorphic to a direct summand of some $N_i^\flat$. 
\end{proof}

\begin{theorem}
    $|\mathcal{MCM}_r(S)| \leq |\mathcal{MCM}(S^\flat)|$ and $|\mathcal{MCM}_l(S^\flat)| \leq |\mathcal{MCM}(S)|$. In addition, $|\mathcal{MCM}_r(S)| = |\mathcal{MCM}_l(S^\flat)|$.
\end{theorem}

\begin{proof}
    Given some $N \in \mathcal{MCM}_r(S)$, then $N$ has a decomposition $N^\flat = \oplus_i M_{ i}$ of indecomposible $\overline{R}-$modules. Therefore, we can conclude that $N \oplus \syz^{S}_1(N) \cong  (N^\flat)^\sharp = \oplus_i M_{ i}^\sharp$, and so by Proposition~\ref{prop:KRS} and by Proposition~\ref{prop_App:M_decom} we can conclude that $N \cong M_{i}^\sharp$ for some $i$. This tells us that $\mathcal{MCM}_r(S)$ is a subset of the set $\{[M^\sharp] \colon [M] \in \mathcal{MCM}(S^\flat)\}$, whose cardinality is smaller than that of $\mathcal{MCM}(S^\flat)$, as it is the image of the map $M \mapsto M^\sharp$. \\
     
    Similarly, if $M \in \mathcal{MCM}_l(S^\flat)$ then $M^\sharp$ decomposes into a direct sum $M^\sharp = \oplus_i N_{i}$, and so by by Proposition~\ref{prop_App:M_flatsharp} we have that $M \oplus \syz_1^{S^\flat} (M) \cong (M^\sharp)^\flat = \oplus_i N_{i}^\flat$. Thus, by Proposition~\ref{prop:KRS} and by Proposition~\ref{prop_App:M_decom}  we have that $M \cong N^\flat_i$ for some $i$, which tells us that $\mathcal{MCM}_l(S^\flat)$ is a subset of $\{[N^\flat] \colon [N] \in \mathcal{MCM}(S)\}$. 
\end{proof}


\bibliographystyle{alpha}
\bibliography{bib}

\end{document}